\documentclass[11pt]{article}
\usepackage[a4paper,margin=25mm]{geometry}
\usepackage[T1]{fontenc}
\usepackage{lmodern}
\usepackage{amsmath,amssymb,amsthm,mathtools}
\usepackage{enumitem,booktabs,array}
\usepackage[expansion=false]{microtype}
\usepackage[hidelinks]{hyperref}
\newtheorem{theorem}{Theorem}[section]
\newtheorem{lemma}[theorem]{Lemma}
\newtheorem{proposition}[theorem]{Proposition}
\newtheorem{corollary}[theorem]{Corollary}
\theoremstyle{definition}
\newtheorem{definition}[theorem]{Definition}
\theoremstyle{remark}
\newtheorem{remark}[theorem]{Remark}
\newcommand{\Sk}{\operatorname{Sk}}
\newcommand{\N}{\mathbb N}
\newcommand{\R}{\mathbb R}
\newcommand{\T}{\mathbb T}
\newcommand{\st}{\operatorname{st}}

\newcommand{\up}{\uparrow}
\newcommand{\dn}{\downarrow}
\newcommand{\Z}{\mathbb Z}
\newcommand{\arch}{\mathrel{\asymp}}
\newcommand{\hs}{\mathbin{\oplus}}
\newcommand{\od}{\mathbin{\odot}}
\newcommand{\EE}{\mathcal E}
\newcommand{\tr}{\operatorname{tr}}
\newcommand{\lm}{\operatorname{lm}}
\newcommand{\lc}{\operatorname{lc}}
\newcommand{\Fr}{\operatorname{Fr}}
\newcommand{\NN}{\mathsf N}
\newcommand{\SSS}{\mathsf S}
\newcommand{\KK}{\mathsf K}
\newcommand{\FF}{\mathsf F}
\newcommand{\GG}{\mathsf G}
\newcommand{\calB}{\mathcal B}
\newcommand{\calA}{\mathcal A}
\newcommand{\calL}{\mathcal L}
\newcommand{\Bsmall}{\mathcal S_0}
\DeclareMathOperator{\coef}{coef}
\newcommand{\Lam}{\Lambda}
\newcommand{\Acal}{\mathcal A}
\newcommand{\Ycal}{\mathcal Y}
\newcommand{\Zcal}{\mathcal Z}
\newcommand{\calW}{\mathcal W}
\newtheorem{hypothesis}[theorem]{Condition}
\numberwithin{equation}{section}
\hypersetup{pdftitle={Finite-tower bounds for Skolem functions}}
\title{Finite-tower bounds for Skolem functions}
\author{Andreas Weiermann,\\
Ghent University, Faculty of Mathematics,\\ Department for Mathematics WE16,\\ Krijgslaan 297, Building S8}
\date{9 September 2026}
\begin{document}
\maketitle
\begin{abstract}
We bound the eventual order types of Skolem functions below finite
exponential towers.
Writing $E_0(u)=u$, $E_{n+1}(u)=2^{E_n(u)}$, and
$\omega_0=1$, $\omega_{k+1}=\omega^{\omega_k}$, the argument gives
\[
 |\Sk_{<E_n(x^m)}|<\omega_{r_n},\qquad
 r_n=2+\frac{n(n+3)}2\quad(n\ge1,\ m\ge2\text{ fixed}).
\]
For triple towers we obtain the sharper bound
$|\Sk_{<E_3(x^m)}|<\omega_{10}$ for every fixed $m\ge1$.
The proof combines comparisons of asymptotic expansions with finite
recursive decompositions and ordinal estimates for ordered sums and
products. The analytic part is developed in the classical field of
logarithmic-exponential series. We prove the required discreteness,
support and truncation statements in this setting. The external
inputs are the series construction of van den Dries--Macintyre--Marker
and the order and ordinal estimates recalled from Berarducci--Mamino.
These bounds imply $|\Sk|=\varepsilon_0$.
\end{abstract}
\noindent
\tableofcontents

\section{Introduction}

In 1956 Skolem introduced the class that now bears his name: the
smallest class of positive functions containing $1$ and $x$ and closed
under addition, multiplication, and exponentiation.  Functions are
compared by eventual growth: $f<g$ means that $f(x)<g(x)$ for all
sufficiently large $x$, and functions agreeing eventually are
identified.  Skolem exhibited a
well-ordered subfamily of order type $\varepsilon_0$, thereby proving
that $\varepsilon_0$ is a lower bound for the eventual order type of
$\Sk$, and conjectured that this lower bound is exact
\cite{Skolem1956}.  Eventual comparison is a linear order; this follows
from Hardy's comparison theory for the larger logarithmico-exponential
class \cite{Hardy1910}.  Richardson subsequently solved the identity
problem for integral exponential functions \cite{Richardson1969}.
Ehrenfeucht then proved,
using Kruskal's tree theorem, that the eventual order on $\Sk$ is in
fact a well-order \cite{Ehrenfeucht1973}.

For a positive function $R$, write $\Sk_{<R}$ for the Skolem functions
eventually smaller than $R$, and use $|A|$ for the order type of an
ordered family $A$, not its cardinality.  We abbreviate finite towers
of ordinal exponentials by
\[
 \omega_0=1,\qquad \omega_{k+1}=\omega^{\omega_k}.
\]

The first general upper bounds were obtained by translating the
formation of Skolem terms into ordinal combinatorics.  Schmidt used the
theory of maximal order types of well-partial orderings developed by
de Jongh and Parikh \cite{deJonghParikh1977} to obtain the
Feferman--Sch\"utte ordinal $\Gamma_0$ as an upper bound
\cite{Schmidt1978}.  Levitz replaced this indirect estimate by a direct
ordinal calculation based on Carruth's arithmetic of order types
\cite{Carruth1942}.  His bound was $\varphi_2(0)$, the least critical
epsilon number, equivalently the least ordinal $\alpha$ satisfying
$\alpha=\varepsilon_\alpha$ \cite{Levitz1978}.  Van den Dries and
Levitz later determined the order type of the fundamental fragment
below $2^{2^x}$, obtaining the exact value \cite{DriesLevitz1984}
\[
 \bigl|\Sk_{<2^{2^x}}\bigr|=\omega^{\omega^\omega}=\omega_3.
\]
This was the first sharp calculation for a fragment extending
substantially beyond the polynomial functions.

Berarducci and Mamino subsequently developed a much finer asymptotic
analysis, using an embedding of the exponential Skolem field into the
surreal numbers.  Among other results, they proved
\[
 \bigl|\Sk_{<2^{x^x}}\bigr|\le\varepsilon_0,
\]
as well as estimates for finite sums and a discreteness theorem for
the possible limits of ratios of Skolem functions \cite{BM}.

We recall here the asymptotic relations that enter this analysis.
For positive germs $f$ and $g$, write $f\preceq g$ if $f=O(g)$,
and write $f\arch g$ if $f\preceq g$ and $g\preceq f$.  The latter
relation says that $f$ and $g$ have the same Archimedean growth rate;
for Skolem functions this is equivalent to $f/g$ tending to a positive
real number.  The finer relation $f\sim g$ means that $f/g$ tends to
$1$.  Thus an Archimedean class fixes the leading order of growth,
while its asymptotic classes distinguish the possible leading constant
factors.  Berarducci and Mamino proved that the asymptotic classes in
each Archimedean class have order type at most $\omega$.  We shall also
use comparisons at a variable precision: for a positive scale $\rho$,
\[
 f\equiv_\rho g
 \quad\Longleftrightarrow\quad
 \log(f/g)=O(\rho).
\]
For $\rho=1$ this is precisely Archimedean equivalence.  Decreasing
$\rho$ makes the comparison finer, while increasing $\rho$ makes it
coarser.  This allows the argument to retain as much of an asymptotic
expansion as is needed at a given stage.

These results provide the order-theoretic basis of the present paper. Earlier work of
Dahn on the limit behaviour of exponential terms provides important
background for this asymptotic viewpoint \cite{Dahn1984}; related
decidability questions for initial fragments were studied by Gurevi\v{c}
\cite{Gurevic1986}.

Our investigation began with two model cases.  We first obtained
$\omega_4$ as an upper bound for the fragment below $2^{x^x}$ and then
$\omega_5$ for the fragment below $2^{2^{2^x}}$.  The next case,
$2^{2^{2^{x^2}}}$, required a more systematic analysis.  Studying it in
depth, and then replacing $x^2$ by $x^m$ for arbitrary fixed $m$, led
to the strategy developed here.  A bound on the leading growth rates
alone is not enough: functions with the same leading behaviour may
still occupy different positions in the eventual order.  The proof
therefore also estimates the order types of families whose members
agree to a prescribed accuracy.  The crucial requirement is that these
estimates remain uniform as the reference function and the accuracy
vary.

Writing
\[
 E_0(u)=u,\qquad E_{n+1}(u)=2^{E_n(u)},
\]
our main result is
\[
 \bigl|\Sk_{<E_n(x^m)}\bigr|<\omega_{r_n},
 \qquad r_n=2+\frac{n(n+3)}2
 \quad(n\ge1,\ m\ge2\text{ fixed}).
\]
The starting point is the elementary identity
$\log(U^b)=b\log U$, which turns exponentiation into multiplication
and relates the comparison of powers to that of their bases and
exponents.  We first bound the possible growth rates of these
logarithms.  We then compare functions more closely, keeping track of
the terms of their asymptotic expansions that exceed a prescribed
error.  The resulting comparisons are organised into finite recursive
decompositions, with simpler comparison problems at each subsequent
step.  Uniform estimates for the ordered families arising in these
decompositions are combined using ordinal bounds for finite sums and
products.  This yields an induction on the height of the exponential
tower, carried out simultaneously for all fixed polynomial degrees
$m$.

We carry out the asymptotic calculations in the classical field of
logarithmic-exponential series of van den Dries, Macintyre and Marker
\cite{DMM97,DMM01}. Its ordered series expansions allow us to retain
precisely the terms that dominate a prescribed error. We give a direct
proof of the required discreteness theorem, valid also for infinite
exponents, and construct the special Hahn subfield containing the
Skolem functions. This develops the analytic argument in ordinary
ordered fields and formal series. In particular, it requires no
embedding into the surreal numbers. The functions and order types in
the statements remain the ordinary Skolem functions and their
eventual order.

Since the finite towers $E_n(x)$
are cofinal in $\Sk$, the bounds for all finite heights yield the upper
bound $|\Sk|\le\varepsilon_0$; together with Skolem's lower bound this
gives $|\Sk|=\varepsilon_0$.

\section{Conventions and proof structure}
Skolem functions are the germs at $+\infty$ generated by
$1,x,+,\cdot$ and exponentiation. We write
$\Sk_{<R}=\{f\in\Sk:f<R\}$ and use the eventual order.
The notation $|A|$ always means order type. Put
\[
 E_0(u)=u,\quad E_{n+1}(u)=2^{E_n(u)},\quad
 \omega_0=1,\quad\omega_{k+1}=\omega^{\omega_k},\quad
 \EE(a)=\omega^{\omega^a}.
\]
Unmarked ordinal operations are ordinary operations; natural sum,
product and iterated product are denoted by $\hs,\od$ and
$\alpha^{\od\beta}$. Set $\N=\{0,1,2,\ldots\}$.
For an ordered additive family $A$, let $\Sigma^+A$ be its
nonempty finite sums, allowing repetition, and put
$\Sigma^0A=\{0\}\cup\Sigma^+A$.

For germs, $f\preceq g$ means $f=O(g)$, $f\prec g$ means
$f=o(g)$, $f\arch g$ means mutual domination, and
$f\sim g$ means $f/g\to1$. Write $[f]$
for the Archimedean class of a positive $f$. Constants in
$O(\cdot)$ are real constants and may depend on the member germ.
The ambient ordered exponential field is the well-based
logarithmic-exponential series field $\T$, with positive infinite
variable $x$. Section~\ref{sf:sec:series} specifies its construction
and the classical facts being used. Skolem germs are interpreted
canonically in $\T$, preserving their order and generating operations.
All auxiliary scales, cutoffs, widths and exponential parameters are
elements of $\T$ unless a smaller domain is specified; real constants
remain real, and ordinal parameters are ordinary ordinals. A reference
required to be a Skolem function is always an actual member of $\Sk$.
The asymptotic comparisons of Skolem germs agree with those of their
images. For field elements, $u=O(v)$ means $|u|\le C|v|$ for some
real $C>0$, while $u=o(v)$ means $|u|<\epsilon|v|$ for every
real $\epsilon>0$. For positive $\rho\in\T$,
\[
 f\equiv_\rho g\iff\log(f/g)=O(\rho),\qquad
 Q_\rho(A)=|A/{\equiv_\rho}|.
\]
Thus $Q_1(A)=|A/\arch|$. The notation $A/O(Y)$ instead denotes
the image in the \emph{additive} ordered group modulo its convex
subgroup $O(Y)$. These two quotient constructions are distinct.
For convex fibres of type at most $\alpha$, indexed by a quotient
of type $\beta$, the ordered-sum bound is
$\alpha\beta\le\alpha\od\beta$.

We first record the order and ordinal estimates used in the counting
argument. We then prove discreteness in an elementary extension of the
real exponential field and develop the support and truncation calculus
in logarithmic-exponential series. The monoid and alphabet estimates
lead to a global counting implication under explicit local hypotheses.
Those hypotheses are proved by finite coarse-tree induction and
successive window tiers. Finally, an induction on tower height,
simultaneous in the fixed polynomial degree, supplies the cap and
leaf bounds.

\section{Order and ordinal inputs}
The following facts are recalled from \cite{BM}; several are classical
results reviewed there. They supply the order-theoretic inputs to the
argument. The separate series-theoretic inputs are listed as
(LE1)--(LE4) in Section~\ref{sf:sec:series}; the discreteness and
special-support statements are proved internally below.
\begin{enumerate}[label=\textup{(BM\arabic*)},leftmargin=4em]
\item $\Sk$ is well ordered. Each member is a finite positive
sum of finite products of components. A component above $x$ is
$A^b$ with $A\ge2$ and component exponent $b\ge x$
(\S1, Remark 10.1 and Proposition 10.2).
Integer values at positive integer arguments follow directly
from the term definition. The resulting discreteness of bounded
Skolem differences is proved in Lemma~\ref{sf:lem:integers}.
Every nonconstant Skolem function is eventually at least $x$,
by induction on a reduced term.
\item Weakly increasing images of products of well orders have
the natural-product bound; finite unions have the natural-sum
bound (Fact 4.1).
\item For a positive well-ordered subset $A$ of an ordered abelian
group, $|A|\le\alpha$, $|A/\arch|\le\beta$, $\alpha\ge2$ imply
\begin{equation}
 |\Sigma^+A|\le(\alpha^\omega)^{\od\beta}.
 \label{ext-sums}
\end{equation}
If $k\ge2$, $\alpha<\omega_{k+1}$ and $\beta<\omega_k$,
this bound is strictly below $\omega_{k+1}$
(Theorem 4.5 and Lemma 4.8).
\item At limit exponents, natural powers equal ordinary powers.
Also, for $u\ge v$,
\begin{equation}
 u\hs v\le u+v\cdot2,\qquad
 \sup_{k<\omega}u\od k=u\cdot\omega.
 \label{ext-ordinals}
\end{equation}
These are Lemma 3.8, Corollary 3.10 and Proposition 3.11.
The finite towers $\omega_k$ are closed under finite natural
sums and products of smaller ordinals.
\item The classical fragment satisfies
$|\Sk_{<2^{2^x}}|\le\omega_3$ and
$|\Sk_{<2^{2^x}}/\arch|\le\omega_2$
(Theorem 13.1). Moreover $\Sk_{<2^x}=\N[x]\setminus\{0\}$,
and $\varepsilon_0\le|\Sk|$ (\S1).
\end{enumerate}
The elementary cases of an empty or singleton summand family
are absorbed by the infinite majorants used below. We use the
usual finite-branching tree principle: an infinite finitely
branching rooted tree has an infinite branch. It follows by
successively choosing a child with infinitely many descendants.

\section{Discreteness in ordered exponential fields}\label{sf:sec:field}

\subsection{Elementary extensions and scaled comparisons}

Fix an elementary extension $F\succeq\R_{\exp}$ and an element
$X\in F$ greater than every real number. Evaluate Skolem terms at $X$.
This identifies $\Sk$ with an ordered subset of $F$, compatibly with
addition, multiplication and exponentiation: an eventual inequality
or equality of terms transfers to $X$ by elementarity, and eventual
totality gives the converse. We subsequently write $x$
for this image of the identity function.

For $v\ne0$, the notation $u=O(v)$ means $|u|\le C|v|$ for some
$C\in\R_{>0}$, and $u=o(v)$ means $|u|<\epsilon|v|$ for every
$\epsilon\in\R_{>0}$. For positive $u,v$, write $u\preceq v$
for $u=O(v)$, $u\prec v$ for $u=o(v)$, and $u\succ v$ for
$v\prec u$. Write $u\arch v$ for mutual $O$-domination, and
$u\sim v$ for $u/v=1+o(1)$.
The logarithm on $F_{>0}$ is the inverse of its exponential.

Every finite $a\in F$ has a unique standard part $\st(a)\in\R$
with $a-\st(a)=o(1)$. Indeed, take the supremum in $\R$ of the real
numbers below $a$. Completeness of $\R$ gives existence, and the
definition gives uniqueness. Standard part preserves sums, products
and order wherever the elements are finite. If $a>0$ is finite and
not infinitesimal, then
\[
 \log a=\log\st(a)+o(1),\qquad
 e^b=e^{\st(b)}+o(1)\quad(b\text{ finite}).
\]
These facts follow from the usual real continuity estimates and
elementarity. We use $u^v=\exp(v\log u)$ for $u>0$.

\begin{lemma}[Bounded Skolem differences]\label{sf:lem:integers}
If $f,g\in\Sk$ and $f-g=O(1)$ in $F$, then $f-g\in\Z$.
\end{lemma}
\begin{proof}
Choose an integer $M>|f-g|$ in $F$. Eventual totality, applied to
$f,g+M$ and to $g,f+M$, together with the order-preserving evaluation,
gives $|f-g|<M$ eventually for the germs. Thus
at all sufficiently large positive integers $n$, the integer
$f(n)-g(n)$ belongs to one fixed finite subset of $\Z$. Some integer
$k$ occurs at arbitrarily large $n$. If $k\ge0$, eventual totality
applied to $f$ and $g+k$ implies $f=g+k$: either strict eventual
inequality would contradict those equalities. If $k<0$, apply the
same argument to $f+(-k)$ and $g$. The case $k=0$ needs no addition.
\end{proof}

\begin{lemma}[Scaled comparisons]\label{sf:lem:scaled}
Let $c\in F$, $c\ge1$.
\begin{enumerate}[label=\textup{(\roman*)},leftmargin=2.5em]
\item The relation $u\arch_c v$ defined by $(u/v)^c\arch1$ is
an equivalence relation with convex classes on $F_{>0}$. It implies
$u\arch v$.
\item If $c$ is infinite and $z>0$, then for every $s\in\R$,
\[
 z^c=e^s+o(1)\quad\Longleftrightarrow\quad
 c(z-1)=s+o(1).
\]
\item If $c,d\ge1$, $d/c=t+o(1)$ with $t\in\R_{>0}$, and
$z^c=r+o(1)$ with $r>0$ real, then $z^d=r^t+o(1)$.
\item If $c$ is finite and $\alpha=\st(c)$, then
$z^c=r+o(1)$, $r>0$, implies $z=r^{1/\alpha}+o(1)$.
\end{enumerate}
\end{lemma}
\begin{proof}
Products, reciprocals and monotonicity of positive powers give (i).
For the final implication in (i), if $u/v$ is infinite, then its
$c$th power is infinite, and if $u/v$ is infinitesimal its $c$th
power is infinitesimal.

For (ii), $z^c=e^s+o(1)$ gives $c\log z=s+o(1)$, so
$\log z=O(1/c)=o(1)$. The transferred real estimate
$e^t=1+t+O(t^2)$ for infinitesimal $t$ gives
$c(z-1)=c\log z+O(1/c)=s+o(1)$. Conversely
$z-1=O(1/c)$ and $\log(1+t)=t+O(t^2)$ give the asserted
limit after exponentiating. For (iii), use
$d\log z=(d/c)(c\log z)=t\log r+o(1)$.
For (iv), divide $c\log z=\log r+o(1)$ by $c=\alpha+o(1)$.
\end{proof}

\subsection{Additive irreducibles}

Call $a\in\Sk$ an \emph{additive irreducible} if it cannot be written
as $u+v$ with $u,v\in\Sk$. This is additive irreducibility; it
does not require multiplicative irreducibility.

\begin{lemma}\label{sf:lem:atoms}
Every member of $\Sk$ is a finite nonempty sum of additive irreducibles.
Every additive irreducible other than $1$ and $x$ admits one of the forms
\[
 uv\quad(u,v\in\Sk,\ u,v\ge x),\qquad
 u^v\quad(u,v\in\Sk,\ u\ge2,\ v\ge x).
\]
\end{lemma}
\begin{proof}
For the first assertion, a least counterexample in the well-order
$\Sk$ would split as $u+v$ with $u,v$ smaller, and both would
already have finite decompositions.

For the second assertion, choose a shortest term representing the
irreducible. Its last operation cannot be addition. If it is a product,
neither factor is $1$. A constant factor $k\ge2$ would express the
product as a sum of $k$ equal Skolem functions, a contradiction.
Thus both factors are nonconstant and at least $x$.
If the last operation is $u^v$, trivial powers have already been
simplified, so $u,v\ge2$. For nonconstant $v$, we have $v\ge x$.
Otherwise $v=k\ge2$ is an integer. The base $u$ must be nonconstant,
since a constant result greater than $1$ is additively reducible.
Now $u^k=u\cdot u^{k-1}$ has both factors at least $x$.
\end{proof}

\begin{lemma}\label{sf:lem:subsequence}
Every sequence in a well-order has an infinite weakly increasing
subsequence. The same is true simultaneously for finitely many
sequences in well-orders.
\end{lemma}
\begin{proof}
If some value occurs infinitely often, take a constant subsequence.
Otherwise choose the least value in the remaining tail and one of
its occurrences, then pass beyond its last occurrence. Repeating
gives a strictly increasing subsequence. Apply this successively to
the finitely many coordinates.
\end{proof}

\subsection{Discreteness for all field exponents}\label{sf:sec:discrete}

The next theorem is a form of \cite[Theorem~11.1]{BM} in an
elementary extension of the real exponential field. Its proof
follows their least-reference and least-integer-bound argument.
We give the details using additive irreducibles, with an actual
Skolem function at every change of reference.

\begin{theorem}[Berarducci--Mamino discreteness in $F$]\label{sf:thm:discrete}
For every $Q\in\Sk$ and every $c\in F$ with $c\ge1$, the set
\[
 D_c(Q)=\{r\in\R_{>0}:(f/Q)^c=r+o(1)
                         \text{ for some }f\in\Sk\}
\]
is well ordered and has no accumulation point in $\R$.
In particular $D_c(Q)\cap(0,R]$ is finite for every real $R>0$,
and $|D_c(Q)|\le\omega$.
\end{theorem}
\begin{proof}
The map $f\mapsto\st((f/Q)^c)$ on those $f$ with positive finite
standard part is weakly increasing. Its image is therefore well
ordered: for a nonempty subset of the image, choose the least
element of its inverse image in $\Sk$.

A well-ordered subset of $\R_{>0}$ has a positive minimum, and so
cannot accumulate at zero. If it has an accumulation point anywhere
else, it has a strictly increasing bounded sequence. Indeed an
infinite bounded subset has order type at least $\omega$; its first
$\omega$ elements give such a sequence. It suffices to rule these out.

Assume a counterexample exists. Choose the least $Q\in\Sk$ for
which a counterexample exists for \emph{some} $c\in F_{\ge1}$.
Thus, for every $P\in\Sk$ with $P<Q$ and every $d\in F_{\ge1}$,
a bounded weakly increasing sequence of positive elements of $D_d(P)$
is eventually constant. We refer to this as the reference induction.

First we record a consequence which does not use any further
minimality: no counterexample at $Q$ can consist entirely of additive
irreducibles. Suppose otherwise, writing
\[
 (a_i/Q)^c=r_i+o(1),\qquad r_0<r_1<\cdots<R.
\]
Then $a_i$ is strictly increasing and $a_i\arch Q$.
The exceptions $1,x$ can occur at most twice. By
Lemma~\ref{sf:lem:atoms} and a subsequence, all remaining $a_i$ have
one of its two forms, with both coordinate sequences weakly
increasing by Lemma~\ref{sf:lem:subsequence}. Replacing the reference
in the displayed ratios by $a_0$ divides every $r_i$ by $r_0$.

If $a_i=u_iv_i$ with $u_i,v_i\ge x$, then
\[
 (a_i/a_0)^c=(u_i/u_0)^c(v_i/v_0)^c.
\]
Both factors are at least $1$ and have bounded positive real standard
parts. Their standard parts are weakly increasing. As
$u_0,v_0\prec a_0\arch Q$, reference induction makes both
sequences eventually constant, contradicting strict increase of $r_i$.

If $a_i=u_i^{v_i}$ with $u_i\ge2$ and $v_i\ge x$, then
\[
 (a_i/a_0)^c\ge a_i/a_0\ge2^{v_i-v_0}.
\]
The left side has standard part uniformly bounded in $i$; enlarging
that real bound by $1$ gives a uniform bound in $F$ as well. Hence
$v_i-v_0$ is bounded by one integer. Lemma~\ref{sf:lem:integers}
and monotonicity show that $v_i$ is eventually constant. After
discarding an initial segment and renumbering, all $v_i=v_0$ and
\[
 (a_i/a_0)^c=(u_i/u_0)^{v_0c}.
\]
Here $u_0\prec Q$, and $v_0c\in F_{\ge1}$. Reference induction
again contradicts strict increase. This proves the assertion about
additive irreducibles for every exponent and every counterexample at $Q$.

We now choose a counterexample of the form
\begin{equation}
 (h_i/Q)^c=r_i+o(1),\qquad 0<r_0<r_1<\cdots<R,
 \label{sf:eq:witness}
\end{equation}
with an additional minimality condition. Let $A$ be the least Skolem
function Archimedean-equivalent to $Q$. For every counterexample
at $Q$, all its members satisfy $h_i\le NA$ for some common
positive integer $N$. To see this, if $h_i/Q\ge1$, then
$h_i/Q\le(h_i/Q)^c<R+1$; if $h_i/Q<1$ there is already such
a bound. Combine this with $Q=O(A)$. Choose the smallest $N$
among all counterexamples at $Q$, allowing the exponent $c$ to vary.
Fix one such witness. Passing to subsequences preserves its bound.

The element $Q$ is the least Skolem member of its $\arch_c$-class.
Otherwise a smaller reference $Q'\arch_c Q$ would multiply the
real sequence in \eqref{sf:eq:witness} by a fixed positive real constant,
contradicting minimality of $Q$.

Only finitely many $h_i$ can be additive irreducibles, by the assertion
already proved. Decompose every remaining $h_i$ into additive irreducibles,
choose one of maximal value and call it $f_i$, and let $g_i$ be the
sum of the others. Then
\begin{equation}
 h_i=f_i+g_i,\quad f_i,g_i\in\Sk,\quad
 f_i\text{ an additive irreducible},\quad f_i\arch h_i\arch Q.
 \label{sf:eq:split}
\end{equation}
The comparison $f_i\arch h_i$ follows because the decomposition
has a finite number of summands for each $i$; no uniform bound on
that number is used. Pass to a subsequence on which both $f_i$ and
$g_i$ are weakly increasing. In particular $f_i\ge A$.

If $c$ is finite, Lemma~\ref{sf:lem:scaled}(iv) changes the witness
to one with exponent $1$, preserving strict increase, boundedness,
and the integer bound $N$. Minimality of $Q$ now implies $A=Q$.
Let
\[
 a_i=\st(f_i/Q)>0,\qquad b_i=\st(g_i/Q)\ge0.
\]
Both sequences are weakly increasing and bounded, and their sum is
the strictly increasing ratio sequence. If all $b_i=0$, the
sequence $f_i$ contradicts the assertion about additive irreducibles.
Otherwise discard an initial segment so that every $b_i>0$.
Then $f_i,g_i\arch Q$, so both are at least $A=Q$, and
\[
 f_i,g_i\le(N-1)A.
\]
If either $a_i$ or $b_i$ were not eventually constant, a strictly
increasing subsequence would be a counterexample at $Q$ with smaller
integer bound. Therefore both are eventually constant, a contradiction.

It remains to treat infinite $c$. Normalize at $h_0$ and put
\[
 (h_i/h_0)^c=R_i+o(1),\qquad s_i=\log R_i.
\]
Then $R_i=r_i/r_0$, and $s_i$ is a strictly increasing bounded
sequence of nonnegative reals. Lemma~\ref{sf:lem:scaled}(ii), followed
by \eqref{sf:eq:split}, gives
\[
 c(f_i/f_0-1)+(c/f_0)(g_i-g_0)
       =s_i(1+g_0/f_0)+o(1).
\]
Both summands on the left are nonnegative and finite. Define
\begin{equation}
 a_i=\st\bigl(c(f_i/f_0-1)\bigr),\qquad
 b_i=\st\bigl((c/f_0)(g_i-g_0)\bigr),\qquad
 q=\st(g_0/f_0)\ge0.
 \label{sf:eq:ab}
\end{equation}
The two sequences are weakly increasing and bounded, and
\begin{equation}
 a_i+b_i=(1+q)s_i.\label{sf:eq:sum-ab}
\end{equation}

The sequence $a_i$ is eventually constant. Indeed
$(f_i/f_0)^c=e^{a_i}+o(1)$. If $f_0<Q$, use reference induction.
If $Q\le f_0$, then $Q\le f_0\le h_0\arch_c Q$ puts $f_0$
in the same $\arch_c$-class by convexity. Thus a nonconstant
$a_i$ would yield a counterexample at $Q$ consisting entirely of
the additive irreducibles $f_i$, which has been excluded.

Set $P=f_0/c$. Then $P\prec Q$, and the definition of $b_i$ gives
\begin{equation}
 g_i-g_0=b_iP+o(P).\label{sf:eq:remainder}
\end{equation}
We prove that $b_i$ is eventually constant by three cases.

If $g_0\prec P$, then $g_i/P=b_i+o(1)$. If every $b_i=0$
there is nothing to prove. Otherwise choose $k$ with $b_k>0$.
For $i\ge k$ the actual Skolem reference $g_k$ satisfies
\[
 g_k\sim b_kP\prec Q,\qquad
 g_i/g_k=b_i/b_k+o(1).
\]
Reference induction with exponent $1$ proves the assertion.
If $g_0\arch P$, write $g_0/P=t+o(1)$ with $t>0$ real.
Then $g_0\prec Q$ and
$g_i/g_0=1+b_i/t+o(1)$, so the same induction applies to $g_0$.
In neither case is $P$ used as a Skolem reference.

Finally suppose $g_0\succ P$, and set $d=g_0/P=(g_0/f_0)c$.
This is an infinite element of $F$. Dividing \eqref{sf:eq:remainder}
by $g_0$ and using Lemma~\ref{sf:lem:scaled}(ii) yields
\[
 (g_i/g_0)^d=e^{b_i}+o(1).
\]
If $g_0<Q$, reference induction applies. Otherwise
$Q\le g_0\le h_0\arch_c Q$, so $g_0\arch_c Q$.
Moreover
\[
 Q\le g_0\le h_0\arch Q,\qquad f_0\arch h_0\arch Q,
\]
so $g_0$ and $f_0$ are both Archimedean-equivalent to $Q$.
Thus $g_0\arch f_0$, and $c/d=\gamma+o(1)$ for some
$\gamma\in\R_{>0}$. Lemma~\ref{sf:lem:scaled}(iii) gives
\[
 (g_i/Q)^c=R_*e^{\gamma b_i}+o(1)
\]
for a fixed $R_*>0$ real. As $f_i\ge A$ and $h_i\le NA$,
we have $g_i\le(N-1)A$. A nonconstant $b_i$ would therefore
contradict the choice of $N$. This completes all three cases.

Both sequences on the left of \eqref{sf:eq:sum-ab} are eventually
constant, whereas its right side is strictly increasing. The
contradiction proves the theorem. Finiteness in each bounded real
interval follows from compactness, including the exclusion of
accumulation at zero, and implies the order-type assertion.
\end{proof}

\begin{remark}
The proof is simultaneous in all $c\in F_{\ge1}$. Restricting the
induction to $c=1$ would not justify the exponents $v_0c$ and
$(g_0/f_0)c$ occurring above. Both are ordinary field elements;
no enlargement of the field is made during the proof.
\end{remark}

\section{The field of logarithmic-exponential series}\label{sf:sec:series}

Take for $F$ the standard field $\T$ of well-based
logarithmic-exponential series over $\R$. Its formal variable $x$
is positive infinite. The construction and the following facts are
classical \cite{DMM97,DMM01}.

\begin{enumerate}[label=\textup{(LE\arabic*)},leftmargin=4.2em]
\item $\T$ is an elementary extension of $\R_{\exp}$, with its
standard exponential and logarithm. The canonical interpretation of
Skolem germs fixes $\R$ and sends the identity germ to formal $x$.
See \cite[Corollary~2.8, (2.11), and (3.10)]{DMM97}, or
\cite[(3.11)--(3.12)]{DMM01}.
\item $\T$ has an ordered monomial group $\mathfrak M$, and every
element has a unique Hahn expansion with real coefficients and
reverse-well-ordered support in $\mathfrak M$. Distinct monomials
belong to distinct Archimedean classes. The monomials $x$ and
$\log x$ belong to $\mathfrak M$.
\item For $y=\log x$, there are increasing monomial groups
$\mathfrak M_h(y)$ and full Hahn fields
\[
 T_h(y)=\R((\mathfrak M_h(y)))\subseteq\T,\qquad
 \mathfrak M_0(y)=y^{\R},
\]
with
\[
 \exp(T_h(y)^\up)\subseteq\mathfrak M_{h+1}(y).
\]
Here $V^\up$ denotes the additive subspace of series supported
strictly above $1$; it includes $0$ and series of either sign.
More explicitly, in the notation of \cite[(2.7)--(2.8)]{DMM01},
we take
\[
 T_h(y)=L_{1,h},\qquad \mathfrak M_h(y)=G_{1,h}.
\]
The transport isomorphism $\eta_1^{-1}$ sends the original
variable $x$ to $y=\log x$, and transports DMM's finite Hahn
fields and their monomial groups to these levels. The exponential
inclusion follows from \cite[(2.1) and Lemma~2.2]{DMM01}
under the same transport.
\item The expansions and their restrictions to subfamilies of the
support agree under the embeddings of finite levels. A Hahn sum is
formed inside a common finite level. For an infinitesimal
$\epsilon$ in such a Hahn field, the families defining
\[
 \exp\epsilon=\sum_{j\ge0}\frac{\epsilon^j}{j!},\qquad
 \log(1+\epsilon)=\sum_{j\ge1}\frac{(-1)^{j+1}}j\epsilon^j
\]
are summable there and give the exponential and logarithm of $\T$.
See \cite[\S1 and (2.4), (2.7)--(2.8)]{DMM01}.
\end{enumerate}

We recall the meaning of summable. The union of supports must be
reverse well ordered, and a fixed monomial must occur in only
finitely many members of the family. Each coefficient of the sum is
then a finite real sum. The usual Hahn--Neumann lemma says that if
$S$ is reverse well ordered and every member of $S$ is below $1$,
its finite products form a reverse-well-ordered set, with finitely
many multiset representations of each monomial. This justifies the
displayed power series in a common Hahn field. We do not close
$\T$ or an increasing union of its subfields under arbitrary sums
across unbounded levels.

\section{The special Hahn subfield for Skolem functions}\label{sf:sec:subfield}

Inside $\T$ define, recursively,
\begin{equation}
 \begin{gathered}
 G_0=x^{\Z},\qquad K_n=\R((G_n)),\\
 G_{n+1}=\{\exp(\gamma)x^{\vartheta+k}:
                  \gamma,\vartheta\in K_n^\up,\ k\in\Z\},\\
 G=\bigcup_{n\ge0}G_n,\qquad K=\bigcup_{n\ge0}K_n.
 \end{gathered}\label{sf:eq:K}
\end{equation}
This is the form of the construction in \cite[Definition~9.1]{BM},
now carried out using (LE1)--(LE4). We verify that all full Hahn
fields in \eqref{sf:eq:K} exist inside $\T$.

\begin{lemma}[Uniform finite levels]\label{sf:lem:levels}
For every $n\ge0$, with $y=\log x$,
\[
 G_n\subseteq\mathfrak M_{n+1}(y),\qquad
 K_n\subseteq T_{n+1}(y).
\]
The groups $G_n$ and the fields $K_n$ increase with $n$. Every
restriction of the support of an element of $K_n$ is again in $K_n$.
\end{lemma}
\begin{proof}
The base follows from $x^k=\exp(ky)$ and (LE3). If the assertion
holds for $n$, then
\[
 \gamma+(\vartheta+k)y\in T_{n+1}(y)^\up
 \qquad(\gamma,\vartheta\in K_n^\up).
\]
Indeed all monomials of $\gamma$ and $\vartheta$ are greater than
$1$, and multiplication by the monomial $y>1$ preserves that
property. Thus each element of $G_{n+1}$ is a monomial in
$\mathfrak M_{n+2}(y)$. A reverse-well-ordered support in the
subgroup $G_{n+1}$ is still such a support in this latter group.
The full Hahn field on $G_{n+1}$ is consequently a subfield of
the single full Hahn field $T_{n+2}(y)$.

Each $G_{n+1}$ is a group: multiply by adding $\gamma,\vartheta,k$,
and invert by changing their signs. The base group is contained in
$G_1$ by setting $\gamma=\vartheta=0$. Inductively the previous
formulas for elements of $G_n$ remain available because
$K_{n-1}\subseteq K_n$. This proves nesting of groups and fields.
Restriction of a Hahn support preserves reverse well ordering, so
remains in the same $K_n$.
\end{proof}

\begin{lemma}[Monomial gap]\label{sf:lem:gap}
The least monomial of $G$ strictly greater than $1$ is $x$.
More precisely, each $g\in G$ is either $x^k$ for some $k\in\Z$,
or satisfies $|\log g|\succ\log x$.
\end{lemma}
\begin{proof}
The assertions hold in $G_0$. Suppose the gap statement holds for
$G_n$, and put $y=\log x$. Since $1<y<x$, we have $y\notin G_n$,
and the cosets $G_n$ and $yG_n$ are disjoint. For
\[
 A=\gamma+y\vartheta+ky,\qquad\gamma,\vartheta\in K_n^\up,
\]
the supports of $\gamma$ and $y\vartheta$ are disjoint. All
their monomials are at least $x$, or at least $xy$, respectively.
Consequently, if either series is nonzero, their sum has a nonzero
leading monomial at least $x$, which cannot be cancelled by $ky$.
It follows that $|A|\succ y$. If $\exp A>1$, this implies
$A>0$ and $\exp A>x$. If both series vanish, then
$\exp A=x^k$, which is greater than $1$ exactly when $k\ge1$.
This proves both assertions at the next stage and hence in the union.
\end{proof}

\begin{lemma}[Exponential closure]\label{sf:lem:expclosure}
The field $K$ is closed under exponential and under restriction of
supports. For every positive $a\in K$,
\[
 \log a\in K+K\log x.
\]
\end{lemma}
\begin{proof}
Write $f\in K_n$ as $f=f^\up+f^0+f^\dn$, according as its
support is above, at, or below $1$. Then
\[
 \exp f=\exp(f^\up)e^{f^0}
                      \sum_{j\ge0}(f^\dn)^j/j!\in K_{n+1}.
\]
The first factor belongs to $G_{n+1}$, and the power series lies
in $K_n$. Support restriction was proved in Lemma~\ref{sf:lem:levels}.

Write $a=r m(1+\epsilon)$ with $r>0$ real, $m\in G$ and
$\epsilon\prec1$ in one $K_n$. The definition of $G$ gives
$m=\exp(\gamma)x^{\vartheta+k}$ at some finite stage. Increase
the common stage to include $a,m,\gamma,\vartheta$. By (LE4),
$q=\log(1+\epsilon)\in K$. Thus
\begin{equation}
 \log a=\gamma+(\vartheta+k)\log x+\log r+q
                            \in K+K\log x.\label{sf:eq:logsupport}
\end{equation}
The field $K$ itself is not logarithmically closed:
$\log x\notin K$ by Lemma~\ref{sf:lem:gap}.
\end{proof}

\begin{theorem}[Positive support of Skolem functions]\label{sf:thm:Sk-support}
The interpretation of every Skolem function in $\T$ belongs to
$K^\up+\N$. Thus it has no monomials below $1$, and its
constant term is a nonnegative integer.
\end{theorem}
\begin{proof}
The positive elements of $K^\up+\N$ contain $1,x$ and are
closed under addition and multiplication, by the Hahn support
rules. We prove closure under $a^b$ for positive
$a,b\in K^\up+\N$. Write $b=t+n$, where $n\in\N$ and
$t\in K^\up$. If $t=0$, the power is an ordinary positive
integer power. Otherwise positivity of $b$ implies $t>0$, and
$t$ is infinite. It suffices to prove $a^t\in K^\up$, except
for the harmless case $a=1$.

If $a=k\ge2$ is an integer, then
$a^t=\exp(t\log k)\in G$ is an infinite monomial. Suppose now
$a$ is nonconstant, and write
\[
 a=r m(1+\epsilon),\quad r>0\text{ real},\quad m\in G,
 \quad m>1,\quad\epsilon\in K,\quad\epsilon\prec1.
\]
Because $a$ has support at least $1$, every monomial of
$\epsilon$ is at least $m^{-1}$. Put
\[
 q=\log(1+\epsilon),\qquad tq=s^\up+s^0+\delta,
 \quad s^0\in\R,\quad\delta\prec1.
\]
These elements belong to a common finite Hahn field. Every monomial
of $q$ is at least $m^{-j}$ for some finite $j$, depending on that
monomial: it arises in one of the powers $\epsilon^j$.
The monomials of $t$ are greater than $1$, so the same assertion
holds for the monomials of $tq$ and of its restriction $\delta$.
It follows that every monomial of $\exp\delta$ is at least
$m^{-N}$ for some finite $N$, depending on the monomial. This
includes the constant monomial, with $N=0$.

The exact factorisation of the power is
\begin{equation}
 a^t=M e^{s^0}\exp\delta,\qquad
 M=\exp(t\log r)m^t\exp(s^\up).\label{sf:eq:power-factor}
\end{equation}
Here $M\in G$. For the middle factor, if
$m=\exp(\gamma)x^{\vartheta+k}$, then
\[
 m^t=\exp(t\gamma)x^{t\vartheta+kt}\in G,
\]
because $t\gamma$ and $t\vartheta+kt$ are purely large.
The other two factors defining $M$ also belong to $G$.

Since $q\prec1$, we have $tq=o(t)$. To see explicitly that
also $s^\up=o(t)$, the case $s^\up=0$ is immediate.
Otherwise $q\ne0$, and the purely large part retains the leading
monomial of $tq$. Hence
\[
 \lm(s^\up)=\lm(tq)=\lm(t)\lm(q)<\lm(t),
\]
because $\lm(q)<1$. Distinct monomials have distinct
Archimedean classes, so this gives $s^\up=o(t)$.
For every fixed integer $N\ge0$,
\[
 \log(M/m^N)
    =(t-N)\log m+t\log r+s^\up>0.
\]
Indeed the first term is positive and asymptotic to $t\log m$;
$\log m$ is infinite, so this term dominates the other two.
Consequently $M/m^N>1$. Every monomial of
\eqref{sf:eq:power-factor} is therefore greater than $1$, by the
lower bounds for the monomials of $\exp\delta$. The entire
expression lies in $K$, so $a^t\in K^\up$. Finally $a^n$ has
support at least $1$, and $a^{t+n}=a^t a^n\in K^\up$.
This proves closure under powers and the theorem by induction on
Skolem terms.
\end{proof}

\begin{remark}
The factor $\exp(s^\up)$ in \eqref{sf:eq:power-factor} is necessary:
$t\log(1+\epsilon)$ need not be infinitesimal when $t$ is infinite.
The proof never expands it as an infinitesimal before extracting
its purely large part. All infinite sums stay in a common finite
Hahn field, as ensured by Lemma~\ref{sf:lem:levels}.
\end{remark}

\section{Logarithmic supports and arbitrary cutoffs}\label{sf:sec:trunc}

Put $\mathcal H=K+K\log x$ and $M_1=G\cup G\log x$.
The set $M_1$ is a union of two cosets, not a monomial group.
The ambient group for products and quotients of its monomials
is $\mathfrak M$.

\begin{proposition}\label{sf:prop:logsupport}
For every $W\in\Sk$, $W>0$, one has $\log W\in\mathcal H$,
with support in $M_1$ and an integer coefficient at $\log x$.
The space $\mathcal H$ is a $K$-module and is closed under
restriction of supports. For every integer $r\ge1$,
\begin{equation}
 M_1\cap[1,x^r]=
 \{1,\log x,x,x\log x,\ldots,x^{r-1},x^{r-1}\log x,x^r\}.
 \label{sf:eq:finite-scales}
\end{equation}
Distinct monomials of $M_1$ differ by a factor at least $\log x$.
\end{proposition}
\begin{proof}
Equation~\eqref{sf:eq:logsupport} applies to $W$ by
Theorem~\ref{sf:thm:Sk-support}. Since the cosets $G$ and
$G\log x$ are disjoint, the coefficient at $\log x$ is exactly
the integer $k$ there: $\vartheta$ has no constant term.
For $b,A,B\in K$, $b(A+B\log x)=bA+(bB)\log x$, which
proves the module assertion. Restricting a support restricts the
two cosets separately and stays in the same two finite Hahn fields.

Lemma~\ref{sf:lem:gap} shows that a monomial in $G$ of at most
polynomial size and at least inverse-polynomial size is $x^k$
for some integer $k$. Applying this both to $g$ and to $g\log x$
gives \eqref{sf:eq:finite-scales}.
Within either coset, a quotient greater than $1$ is at least $x$.
Between the cosets such a quotient has the form $g\log x$ or
$g/\log x$. In the first case, $g<1$ would imply
$g\log x\le(\log x)/x<1$, so the quotient is at least
$\log x$. In the second case $g>1$, so the quotient is at least
$x/\log x>\log x$.
\end{proof}

\begin{proposition}[Truncation at a positive series]\label{sf:prop:trunc}
Let $s>0$ belong to $\T$. For $F\in\T$, let
$\tr_{\succ s}F$ retain precisely those support monomials
$m$ with $m/s$ infinite. This operation is linear and weakly
increasing, and preserves $K$ and $\mathcal H$. Moreover
\begin{align}
 F-\tr_{\succ s}F&=O(s),\\
 \tr_{\succ s}F=\tr_{\succ s}F'
       &\quad\Longleftrightarrow\quad F-F'=O(s).
 \label{sf:eq:trunc-equivalence}
\end{align}
If $s\in\mathfrak M$, then
\[
 F=\tr_{\succ s}F+\coef_s(F)s+o(s).
\]
For $h\in\mathfrak M$ there is also the identity
\[
 \tr_{\succ hs}(hF)=h\tr_{\succ s}F.
\]
\end{proposition}
\begin{proof}
Write $s=a m_s(1+\epsilon)$, with $a>0$ real,
$m_s=\lm(s)$ and $\epsilon\prec1$. Then $s\arch m_s$, and
for any monomial $m$,
\[
 m\succ s\quad\Longleftrightarrow\quad m>m_s.
\]
The retained support is an initial segment of the original Hahn
support, so its restriction lies in the same finite level.
The remaining support is at most $m_s$ and therefore gives
an $O(s)$ remainder. Conversely a nonzero series supported
strictly above $m_s$ has a leading term dominating $s$.
Linearity and this observation give \eqref{sf:eq:trunc-equivalence}.
If $F\ge0$, either the truncation is zero or it retains the
positive leading term; linearity therefore gives monotonicity.
The coefficient formula and shift identity follow directly by
separating the boundary monomial and multiplying supports by $h$.
For $A,B\in K$ one has explicitly
\[
 \tr_{\succ s}(A+B\log x)
 =\tr_{\succ s}A+(\log x)\tr_{\succ s/\log x}B\in\mathcal H.
\]
\end{proof}

\begin{corollary}\label{sf:cor:gaps}
For $f,g\in\Sk$,
\[
 f\prec g\ \Longrightarrow\ f=O(g/x),\qquad
 f\sim g\ \Longleftrightarrow\ f-g=O(g/x).
\]
Also $f-g=o(1)$ implies $f=g$.
\end{corollary}
\begin{proof}
For the first assertion compare the leading monomials of $f,g$ in
$G$ and use Lemma~\ref{sf:lem:gap}. For the second, apply the
same gap to $f-g\in K$ when $f-g=o(g)$; the reverse implication
uses $x^{-1}\prec1$. Finally $f-g\in K^\up+\Z$ by
Theorem~\ref{sf:thm:Sk-support}, so a nonzero difference cannot
be infinitesimal.
\end{proof}

\section{Coefficient sets at a fixed prefix}\label{sf:sec:coefficients}

We call a real set \(A\) \emph{\(\N\)-like} if
\begin{equation}
 A\cap(-\infty,R]\quad\hbox{is finite for every }R\in\R.
 \label{f-eq-2.4}
\end{equation}
Such a set is bounded below and has order type at most \(\omega\).
Finite unions and fixed finite sums of \(\N\)-like sets are
\(\N\)-like.  The same is true of all finite sums from a positive
\(\N\)-like set: the set has a positive minimum, so a bounded sum
has bounded length and uses only finitely many possible summands.

\begin{proposition}[Coefficients of Skolem functions]\label{sf:prop:coef}
Let $s>0$ and $U\in\T$ with $U\succ s$. Suppose a nonempty
family $\mathcal A\subseteq\Sk$ satisfies
\[
 f=U+a_f s+o(s),\qquad a_f\in\R\quad(f\in\mathcal A).
\]
Then its coefficient set $\{a_f:f\in\mathcal A\}$ is $\N$-like.
The same holds for a family $f=a_f s+o(s)$ with all $a_f>0$.
\end{proposition}
\begin{proof}
Choose an actual member $f_0\in\mathcal A$, with coefficient
$a_0$. Then $f_0\succ s$ and $d=f_0/s$ is infinite. Exactly,
\[
 d(f/f_0-1)=(f-f_0)/s=a_f-a_0+o(1).
\]
Lemma~\ref{sf:lem:scaled}(ii) gives
$(f/f_0)^d=e^{a_f-a_0}+o(1)$. These positive real values belong
to $D_d(f_0)$, which is finite in every bounded positive interval
by Theorem~\ref{sf:thm:discrete}. Coefficients $a_f\le R$ map
injectively into $(0,e^{R-a_0}]$, proving the assertion even
when the coefficients have either sign. If $U=0$ and $a_f>0$,
use $f/f_0=a_f/a_0+o(1)$ and $D_1(f_0)$ instead.
\end{proof}

\begin{proposition}[Coefficients of logarithms]\label{sf:prop:logcoef}
Let $0<s\le1$ belong to $\T$. Suppose a nonempty family
$\mathcal W\subseteq\Sk$ has a fixed prefix $L\in\T$ with
\[
 \log W=L+c_Ws+o(s),\qquad c_W\in\R.
\]
Then $\{c_W:W\in\mathcal W\}$ is $\N$-like.
\end{proposition}
\begin{proof}
Choose an actual $W_0$ in the family, with coefficient $c_0$.
Subtract the two expansions and divide by $s$ to obtain
\[
 (W/W_0)^{1/s}=e^{c_W-c_0}+o(1).
\]
Here $1/s\in\T_{\ge1}$, and the reference $W_0$ is a Skolem
function. Apply Theorem~\ref{sf:thm:discrete} and the same bounded
interval argument as in Proposition~\ref{sf:prop:coef}.
\end{proof}

\begin{proposition}[Closure of the parameter domain]\label{sf:prop:parameters}
Finite combinations of elements of $\T$ using field operations,
exponential, logarithm, real powers, absolute value, minimum and
maximum remain in $\T$, whenever the displayed operations are
defined. Support restrictions and leading monomials of its series
also belong to $\T$. In particular this field contains all
auxiliary parameters used in the local and window recursions below.
\end{proposition}
\begin{proof}
The field and exponential-logarithmic assertions are (LE1).
A minimum or maximum is one of its arguments. Support restrictions
remain in the finite Hahn level of the original series by (LE4),
and leading monomials lie in $\mathfrak M\subseteq\T$.
For the parameters used later, this applies first to
\[
 \rho_0=2^{-x},\quad \ell=\log H,\quad
 \theta=\log\log x,\quad \kappa=C\theta^2,\quad H^\kappa,
\]
and then to each recursive change, such as
$q+\log b$, $\sigma/b$, $e^{-p}$ and
\[
 S=\varepsilon RV,\quad S^\#=\min(xS,V),\quad
 \eta=\varepsilon V/D,\quad
 Y=e^{-p-z}V,\quad Z=e^zV.
\]
Centers chosen from a frame decomposition remain actual Skolem
functions. Every branch uses only finitely many such operations.
The limit steps of the ordinal estimates take suprema of order
types in ordinary set theory; they do not construct sums of series
across unbounded Hahn levels.
\end{proof}

\section{The monoid and the logarithmic alphabet}

\begin{definition}
For $B\subseteq\Sk$, let $M(B)$ be the family of finite products of
elements of $B$, including the empty product $1$.
\end{definition}

\begin{lemma}[Monoid bound]\label{g-lem:monoid}
Let $B\subseteq\Sk$,
$\alpha\ge\max(2,|B|)$, and $\beta\ge\max(1,|B/\arch|)$.
Then
\[
 |M(B)|\le 1+(\alpha^\omega)^{\od\beta},\qquad
 |M(B)/\arch|\le |M(B)|.
\]
In particular, for $n\ge2$,
\[
 |B|<\omega_{n+1},\quad |B/\arch|<\omega_n
 \quad\Longrightarrow\quad |M(B)|<\omega_{n+1}.
\]
\end{lemma}
\begin{proof}
Put $L=\{\log b:b\in B,\ b>1\}$.  Every such $b$ is eventually
at least $2$.  If $b_1\arch b_2$, then
$\log b_1-\log b_2=O(1)$; since both logarithms are bounded below by
$\log2$, they have the same Archimedean class.  The weakly increasing
map $b\mapsto[\log b]$ therefore factors through $B/\arch$.
Consequently $|L|\le\alpha$ and $|L/\arch|\le\beta$.
Logarithm identifies $M(B)$ with $\{0\}\cup\Sigma^+L$.
Apply (BM3), and then use the monotone quotient projection.
The last assertion follows from (BM3) and closure under addition of
smaller ordinals.  Empty $L$ causes no difficulty.
\end{proof}

\begin{definition}[Alphabet]
Let $H>2$, $T=2^H$, and $B=\Sk_{<H}$.
Let $\Lam(T)$ consist of the Archimedean classes of $\log x$,
of $\log W$ for $1<W<T$, $W\in\Sk$, and of
$b\log W$ for $b\in B$, $1<W<T$, $W\in\Sk$.
Repeated classes are included only once.
\end{definition}

\begin{lemma}[Cap-derived alphabet]\label{g-lem:alphabet}
For every $W\in\Sk$ with $1<W<T$ there are $c\in M(B)$ and
$\delta\in\{0,1\}$ such that
\[
 \log W\arch c(\log x)^\delta.
\]
Moreover
\[
 \boxed{\quad |\Lam(T)|\le |M(B)/\arch|\od2.\quad}
\]
\end{lemma}
\begin{proof}
Reduce a term by deleting the identities $1^g=1$, $1g=g1=g$,
and $g^1=g$.  Every proper subterm of a reduced term $W>1$
is smaller than $W$.  For a nontrivial power $U^b$, this uses
$U,b\ge2$, $U<U^b$, and $b<2^b\le U^b$.
These inequalities hold eventually: by induction on terms a
nonconstant Skolem function is eventually at least $x$, and a constant
Skolem function is a positive integer.

Induct on the reduced term.  A constant $W>1$ has
$[\log W]=[1]$, and $W=x$ has $[\log W]=[\log x]$.
If $W=U+V$ is nonconstant, put $M=\max(U,V)>1$.  Then
\[
 \log M\le\log(U+V)\le\log M+\log2.
\]
Thus $[\log(U+V)]=[\log M]$.  The constant case was already handled.
If $W=UV$, the logarithm is the sum of the two nonnegative
logarithms, so its class is the class of the larger nonzero one.

Finally suppose $W=U^b<T$.  Since $U\ge2$, the inequality
$2^b\le U^b<2^H$ gives $b<H$.  The induction hypothesis for $U$
therefore yields
\[
 \log W=b\log U\arch (bc_U)(\log x)^{\delta_U},
 \qquad bc_U\in M(B).
\]
Only $b<H$ is used: there is no assertion that the base $U$ is $<H$.
The term is finite, so the accumulated product is finite.

For the count, every atom also has the form
$[b\log W]=[(bc)(\log x)^\delta]$, and $bc\in M(B)$.
The empty product captures $[1]$ and $[\log x]$.
The weakly increasing map
\[
 (M(B)/\arch)\times\{0,1\}\longrightarrow
 \{\text{Archimedean classes of positive germs}\},\qquad
 ([c],\delta)\longmapsto[c(\log x)^\delta]
\]
has an image containing the whole alphabet.  Apply (BM2).
This also proves that the alphabet is well ordered.
\end{proof}

\begin{corollary}\label{g-cor:alphabet}
For $n\ge2$, if $|B|<\omega_{n+1}$ and $|B/\arch|<\omega_n$,
then $|\Lam(T)|<\omega_{n+1}$.  In particular
$|\Lam(T)|\le\omega^\lambda$ for some $\lambda<\omega_n$.
\end{corollary}
\begin{proof}
Use Lemmas~\ref{g-lem:monoid} and \ref{g-lem:alphabet} and closure of
$\omega_{n+1}$ under natural products of strictly smaller ordinals.
Since $\omega_n$ is a limit ordinal, the powers $\omega^\lambda$,
$\lambda<\omega_n$, are cofinal in $\omega_{n+1}$.
\end{proof}

\section{Frames and the local estimates for the global ledger}

An unrestricted frame is an expression
\[
 D=P\prod_{j=1}^r W_j^{b_j},\qquad
 P\in\N[x]\setminus\{0\},\quad
 b_j\in\Sk,\quad b_j\ge x,\quad
 W_j\in\Sk,\quad W_j>1.
\]
The case $r=0$ is allowed. Its logarithmic terms $b_j\log W_j$
are called atoms. By (BM1), every Skolem function is a finite
positive sum of such frames.

Now fix $H>2$, $T=2^H$, and $B=\Sk_{<H}$, with $x<H$.
A frame below the cap means an unrestricted frame with value $<T$.
Every factor and every base of such a frame is $<T$. Moreover,
for each label, $2^{b_j}\le W_j^{b_j}\le D<T$ implies $b_j<H$,
so its labels automatically belong to $B$.
In the arguments with a fixed cap, \emph{frame} means a frame below
that cap; the coefficient induction, which ranges over all of
$\Sk$, uses unrestricted frames.

Set $\rho_0=2^{-x}$.

For $\xi\in\Lam(T)$ put
\[
 C_\xi=\{W\in\Sk_{<T}:W=1\text{ or }[\log W]\le\xi\},
 \qquad C_{<\xi}=\bigcup_{\mu<\xi}C_\mu.
\]
Define $\Ycal_\xi$ to be the set of frames whose values belong to
$C_\xi$, and $\Zcal_\xi$ to be their products with $P$ stripped.
Let $\Acal_\xi$ consist of their atoms.  In particular,
\begin{equation}
 \Ycal_\xi\subseteq C_\xi,\qquad
 \Ycal_\xi\subseteq(\N[x]\setminus\{0\})\cdot\Zcal_\xi.
 \label{g-eq-3.1}
\end{equation}
The second inclusion is not asserted to be an equality.

\begin{lemma}[Descent and initial segments]\label{g-lem:descent}
Every member of $C_\xi$ is a finite positive sum of elements of
$\Ycal_\xi$.  Every atom of such a frame has class at most $\xi$,
and its base belongs to $C_{<\xi}$.  For every $\rho$,
\[
 Q_\rho(C_{<\xi})=\sup_{\mu<\xi}Q_\rho(C_\mu),\qquad
 |C_{<\xi}|=\sup_{\mu<\xi}|C_\mu|.
\]
\end{lemma}
\begin{proof}
Each frame $D$ in a positive decomposition of $W$ satisfies $D\le W$.
Each atom is positive and at most $\log D$.
Since $b\ge x$,
$\log W_j\prec b\log W_j$, so its class is strictly smaller
than the atom class.  This proves the descent.
The $C_\mu$ are initial segments and form an increasing chain.
Convex quotient projections preserve the initial-segment property.
\end{proof}

\begin{hypothesis}[Uniform local fibre estimates]\label{g-hyp:fibres}
Let $g,n$ be ordinals and put
\[
 \Theta=\omega^{\omega^{g+1}},\qquad
 \Phi=\omega^{\omega^{n+1}}.
\]
The following local inequalities will be proved in
Theorem~\ref{local-roots}; they are recorded here to state the
global counting implication:
\begin{enumerate}[label=\textup{(H\arabic*)},leftmargin=3.5em]
\item For every $b\in B$ and every $\equiv_{\rho_0}$-class
$E$ of $\Sk_{<T}$, one has $Q_{1/b}(E)<\Theta$.
\item For every Archimedean class $E$ of $\Sk_{<T}$,
one has $Q_{\rho_0}(E)<\Phi$.
\end{enumerate}
Both bounds are uniform in the class and in all indicated labels.
They imply the same bounds on intersections with initial fragments.
\end{hypothesis}

In (H1), $1/b\succeq\rho_0$ is a trivial case, since the relation is
then coarser on $E$.  Otherwise the relations refine one another, and
the ordered-sum bound on the convex fibres gives
\begin{equation}
 Q_{1/b}(C_\xi)\le Q_{\rho_0}(C_\xi)\od\Theta.
 \label{g-eq-3.2}
\end{equation}
Every member and its dominant frame have the same Archimedean class:
if there are $k$ frames and $D$ is the largest, then
$D\le W\le kD$.  Consequently (H2) gives
\begin{equation}
 Q_{\rho_0}(C_\xi)\le |\Ycal_\xi/\arch|\od\Phi.
 \label{g-eq-3.3}
\end{equation}
Thus the formulation of (H2) is precisely the dominant-frame fibre
condition, expressed without an auxiliary choice of decomposition.

\section{The global ordinal ledger}

\begin{theorem}[One-step counting with explicit hypotheses]\label{g-thm:ledger}
Suppose
\[
 |B|\le\omega^{\omega^a},\qquad
 |\Lam(T)|\le\omega^\lambda,\qquad \lambda\ge1,
\]
and suppose (H1)--(H2) hold with $g,n$.  Define
\begin{equation}
 \begin{aligned}
 \Gamma&=\max(a,g+1,n+1,\lambda+4),\\
 \bar\gamma&=\Gamma+\omega^{\lambda\cdot2+1},\\
 \bar\tau&=\bar\gamma+1+\omega^\lambda,\qquad
 d=\max(\bar\gamma,g+1)+1.
 \end{aligned}
\label{g-ledger-parameters}\end{equation}
Then
\[
 \begin{aligned}
 Q_{\rho_0}(\Sk_{<T})&\le\omega^{\omega^{\bar\gamma}},&
 |\Sk_{<T}/\arch|&\le\omega^{\omega^{\bar\gamma}},\\
 |\Sk_{<T}|&\le\omega^{\omega^{\omega^{\bar\tau}}},&
 Q_{1/b}(\Sk_{<T})&\le\omega^{\omega^d}\quad(b\in B).
 \end{aligned}
\]
\end{theorem}

\begin{lemma}[Finite class factorizations]\label{g-lem:factor}
For a fixed atom class $\mu$ there are finitely many pairs
$([b],[\log W])$ realizing $[b\log W]=\mu$.
\end{lemma}
\begin{proof}
The label classes and the base-logarithm classes are well ordered:
they are monotone images of subsets of $\Sk$.
Positive Archimedean classes form an ordered multiplicative group.
Infinitely many distinct factorizations of $\mu$ would give an
increasing sequence of distinct first factors, and therefore a
decreasing sequence of second factors.  This contradicts well ordering.
\end{proof}

\begin{lemma}[Sector bound]\label{g-lem:sector}
Assume (H1), and suppose
$Q_{\rho_0}(C_\mu)\le\omega^{\omega^{\gamma_\mu}}$ for $\mu<\xi$.
Put
\[
 \gamma_{<\xi}=\sup_{\mu<\xi}\gamma_\mu,\qquad
 \gamma'=\max(\Gamma,\gamma_{<\xi}),\qquad \sup\varnothing=0.
\]
Then
\[
 \begin{aligned}
 |\Acal_\xi/O(1)|&\le\omega^{\omega^{\gamma'+1}},\\
 |\Zcal_\xi/\arch|&\le\omega^{\omega^{\gamma'+\lambda+3}},\\
 |\Ycal_\xi/\arch|&\le
       \omega^{\omega^{\gamma'+\lambda+3}\cdot2}.
 \end{aligned}
\]
\end{lemma}
\begin{proof}
Partition atoms by their Archimedean class $\mu\le\xi$, and then by
the finitely many class factorizations in Lemma~\ref{g-lem:factor}.
Fix a label class with representative $b_0$.
For every label $b\arch b_0$, replacing a base by another base in
the same $\equiv_{1/b_0}$-class changes $b\log W$ by $O(1)$.
Choose the least representatives of these convex classes in
$C_{<\xi}$.  The map from the label and the chosen representative
to $b\log W\bmod O(1)$ is increasing in each argument.
Its image contains the block under consideration; equality with the
whole image is unnecessary.  Hence (BM2), \eqref{g-eq-3.2}, and
Lemma~\ref{g-lem:descent} give
\begin{equation}
 |\text{block}|
 \le\omega^{\omega^a}\od
      \omega^{\omega^{\gamma_{<\xi}}}\od
      \omega^{\omega^{g+1}}
 \le\omega^{\omega^{\gamma'}\cdot3}.
 \label{g-eq-4.1}
\end{equation}
Representatives equal to $1$ may be discarded: an atom base $>1$
cannot belong to its $\equiv_{1/b_0}$-class, since $b_0\ge x$.
For the same reason all atoms considered are infinite.

A finite union of the blocks in \eqref{g-eq-4.1} is bounded by
$\omega^{\omega^{\gamma'}\cdot3+1}$, uniformly in the finite number
of factorizations.  The atom-class blocks are convex in
$\Acal_\xi/O(1)$: atom classes separated by an infinite ratio
cannot be identified by a bounded change.
There are at most $\omega^\lambda$ such blocks.  Their ordered sum
is bounded by
\[
 \omega^{\omega^{\gamma'}\cdot3+1}\cdot\omega^\lambda
 =\omega^{\omega^{\gamma'}\cdot3+1+\lambda}
 \le\omega^{\omega^{\gamma'+1}},
\]
because $\gamma'\ge\lambda+4$ implies $\lambda<\omega^{\gamma'}$.

Logarithm identifies products modulo $\arch$ with their sums of atoms
modulo $O(1)$.  The latter atoms are positive in the ordered group
of germs modulo bounded germs.  Their Archimedean classes in this
group agree with their original classes, because the atoms are infinite.
Use (BM3), with the alphabet bound $\omega^\lambda$.
Since $\lambda\ge1$, $\omega^\lambda$ is a limit ordinal, and (BM4)
gives
\[
 \begin{aligned}
 |\Zcal_\xi/\arch|
 &\le 1+
  \bigl((\omega^{\omega^{\gamma'+1}})^\omega
                \bigr)^{\od\omega^\lambda}\\
 &=1+\omega^{\omega^{(\gamma'+2)+\lambda}}
 \le\omega^{\omega^{\gamma'+\lambda+3}}.
 \end{aligned}
\]
The initial $1$ accounts for the empty product and is absorbed by
the infinite limit majorant.
Finally $|\N[x]\setminus\{0\}|=\omega^\omega$.
Apply the quotient of the increasing map $(P,Z)\mapsto PZ$
to \eqref{g-eq-3.1} and absorb this polynomial factor into the displayed bound.
\end{proof}

\begin{proof}[Proof of Theorem~\ref{g-thm:ledger}: fine quotients]
Define recursively on the alphabet
\begin{equation}
 \gamma_\xi=
 \max\bigl(\Gamma,\sup_{\mu<\xi}\gamma_\mu\bigr)\hs(\lambda+4).
 \label{g-eq-4.2}
\end{equation}
With $\gamma'$ as in Lemma~\ref{g-lem:sector}, \eqref{g-eq-3.3} implies
\[
 \begin{aligned}
 Q_{\rho_0}(C_\xi)
 &\le |\Ycal_\xi/\arch|\od\Phi\\
 &\le\omega^{(\omega^{\gamma'+\lambda+3}\cdot2)
                         \hs\omega^{n+1}}\\
 &\le\omega^{\omega^{\gamma'+\lambda+3}\cdot3}
 \le\omega^{\omega^{\gamma'+\lambda+4}}
 \le\omega^{\omega^{\gamma_\xi}}.
 \end{aligned}
\]
Here $n+1\le\Gamma\le\gamma'$, and ordinary sum is bounded by
natural sum.

Put $\delta=\lambda\cdot2+8$.
For the rank $r$ of $\xi$ in the alphabet, induction gives
\[
 \gamma_\xi\le\Gamma+\delta\cdot(r+1).
\]
Indeed the earlier supremum is at most $\Gamma+\delta r$,
and (BM4) bounds its natural sum with $\lambda+4$ by
$\Gamma+\delta r+(\lambda+4)\cdot2
 \le\Gamma+\delta(r+1)$.
Since $r<\omega^\lambda$ and $\delta<\omega^{\lambda+1}$,
\begin{equation}
 \delta(r+1)
 <\omega^{\lambda+1}\cdot\omega^\lambda
 =\omega^{(\lambda+1)+\lambda}
 \le\omega^{\lambda\cdot2+1}.
 \label{g-eq-4.3}
\end{equation}
Thus $\gamma^*:=\sup_\xi\gamma_\xi\le\bar\gamma$.
Taking the increasing union of the initial fragments proves
the asserted fine-quotient bound.  The Archimedean quotient is coarser.
\end{proof}

\begin{proof}[Proof of Theorem~\ref{g-thm:ledger}: totals and deep quotients]
Put $\tau_0=\bar\gamma+1$ and define
\begin{equation}
 \tau_\xi=\max\bigl(\tau_0,\sup_{\mu<\xi}\tau_\mu\bigr)+1.
 \label{g-eq-4.4}
\end{equation}
We prove $|C_\xi|\le\omega^{\omega^{\omega^{\tau_\xi}}}$.
Write $\tau'=\max(\tau_0,\sup_{\mu<\xi}\tau_\mu)$ and $E=\omega^{\tau'}$.
Then $E>\omega^{\bar\gamma}$, $E>\lambda$, and $E$ is an infinite
additively indecomposable ordinal.

By class descent, $\Acal_\xi$ is a subset of the image of
$B\times(C_{<\xi}\setminus\{1\})$ under $(b,W)\mapsto b\log W$.
The induction hypothesis and (BM2) therefore give
\[
 |\Acal_\xi|
 \le\omega^{\omega^a}\od\omega^{\omega^E}
 \le\omega^{\omega^E\cdot2}.
\]
Applying (BM3)--(BM4) to the atoms, whose class alphabet has type
at most $\omega^\lambda$, yields
\[
 \begin{aligned}
 |\Zcal_\xi|
 &\le1+\bigl((\omega^{\omega^E\cdot2})^\omega
                                 \bigr)^{\od\omega^\lambda}\\
 &=1+\omega^{\omega^{(E+1)+\lambda}}
 \le\omega^{\omega^{E\cdot3}},\\
 |\Ycal_\xi|&\le\omega^{\omega^{E\cdot3}\cdot2}.
 \end{aligned}
\]
The slack in the exponent avoids any endpoint issue for the empty
product or the polynomial factor.

Crucially $\Ycal_\xi\subseteq C_\xi$, so
\begin{equation}
 |\Ycal_\xi/\arch|
 \le Q_{\rho_0}(C_\xi)
 \le\omega^{\omega^{\bar\gamma}}
 \le\omega^E.
 \label{g-eq-4.5}
\end{equation}
Every member of $C_\xi$ is a positive sum of these frames.
A second application of (BM3)--(BM4), now using the limit exponent
$\omega^E$, gives
\[
 \begin{aligned}
 |C_\xi|
 &\le\bigl((\omega^{\omega^{E\cdot3}\cdot2})^\omega
                                      \bigr)^{\od\omega^E}\\
 &=\omega^{\omega^{(E\cdot3+1)+E}}
 =\omega^{\omega^{E\cdot4}}
 <\omega^{\omega^{E\cdot\omega}}
 =\omega^{\omega^{\omega^{\tau'+1}}}.
 \end{aligned}
\]
The finite $1$ is absorbed in $(E\cdot3+1)+E$, since $E$ is
infinite and additively indecomposable.
Recursion \eqref{g-eq-4.4} gives
$\tau_\xi\le\tau_0+\operatorname{rank}(\xi)+1$.
The increasing union of the $C_\xi$ consequently has the asserted
total bound with $\bar\tau=\tau_0+\omega^\lambda$.

Finally \eqref{g-eq-3.2}, the fine bound, and continuity on the initial fragments
give
\[
 Q_{1/b}(\Sk_{<T})
 \le\omega^{\omega^{\bar\gamma}}\od\omega^{\omega^{g+1}}
 \le\omega^{\omega^{\max(\bar\gamma,g+1)+1}}.
 \qedhere
\]
\end{proof}

\section{Support and finite representations}

\begin{lemma}[Logarithmic support and truncation]\label{support-expanded}
Every $\log W$, $W\in\Sk$, belongs to
$\mathcal H=K+K\log x$, has support in $M_1=G\cup G\log x$,
and has an integer coefficient at $\log x$. The same support
restriction holds for atoms $b\log W$ and finite sums of atoms
and polynomial logarithms.
For every positive $s\in\T$, truncation retaining precisely the
monomials $m\succ s$ is linear and weakly increasing, and
\[
 \tr_{\succ s}F=\tr_{\succ s}F'
       \quad\Longleftrightarrow\quad F-F'=O(s).
\]
It preserves $\mathcal H$ and satisfies the monomial shift
identity of Proposition~\ref{sf:prop:trunc}. Products and coefficients
are well defined for infinite supports. For $f,g\in\Sk$,
\begin{equation}
 f\prec g\quad\Longrightarrow\quad f=O(g/x).
 \label{l-eq-2.4}
\end{equation}
\end{lemma}
\begin{proof}
The support and module assertions are
Proposition~\ref{sf:prop:logsupport}, and the truncation assertions
are Proposition~\ref{sf:prop:trunc}. All series in any fixed
calculation lie in one finite Hahn level. The Hahn product and
summability rules in (LE4) give the coefficient assertions.
Equation~\eqref{l-eq-2.4} is Corollary~\ref{sf:cor:gaps}.
A zero truncation is omitted when forming a positive summand set.
\end{proof}

\begin{lemma}[Finite representations]\label{l-lem:representations}
If $A$ is a well-ordered positive subset of an ordered abelian group,
each fixed element has only finitely many representations as a
finite multiset sum of elements of $A$.
\end{lemma}
\begin{proof}
The set $\Sigma^0 A=\{0\}\cup\Sigma^+A$ is well ordered by \eqref{ext-sums}.
For a fixed sum $t$, infinitely many possible summand values would
give an increasing sequence $a_i$ with
$t-a_i\in\Sigma^0A$ strictly decreasing.  There are therefore finitely
many possible summand values.  The multiplicity of any one value $a$
is also bounded: otherwise $t-na\in\Sigma^0A$ gives a decreasing
sequence.  There are consequently finitely many multisets.
\end{proof}

Both the atom family and its nonzero truncations are well ordered:
the atom family is an increasing image of a product of two well
orders.  The same is true after adjoining the polynomial logarithms
$\log P$.  Lemma~\ref{l-lem:representations} therefore applies to a
fixed logarithmic truncation of frames. Separately, the positive
frame family is well ordered, so the lemma applies to fixed
additive truncations of sums of frames as well.  We call each resulting finite list of prescribed nonzero
summand truncations a \emph{cell}.  A cell has finitely many slots,
with multiplicities included.

\section{Finite polynomial-logarithmic scales}

\begin{lemma}[Polynomial-logarithmic scales]\label{f-lem:scales}
For every fixed integer \(r\ge1\),
\begin{equation}
 M_1\cap[1,x^r]
 =\{1,\log x,x,x\log x,\ldots,
                         x^{r-1},x^{r-1}\log x,x^r\}.
 \label{f-eq-2.3}
\end{equation}
In particular there are \(2r\) scales in this interval strictly
above \(1\).
\end{lemma}
\begin{proof}
This is the finite-scale formula of
Proposition~\ref{sf:prop:logsupport}; its displayed list has
$2r$ members strictly greater than $1$.
\end{proof}

\section{A coefficient induction through the shallow scales}

We first isolate the product calculation used in every coefficient
slot.  The fixed leading prefix will be retained throughout the calculation.

\begin{lemma}[One coefficient of an atom]\label{f-lem:atom}
Let
\[
 a=bL,\qquad b\in\Sk,\quad b\ge x,\quad L=\log W,\quad
                  1<W\in\Sk,
\]
and fix \(v\in M_1\) and a nonzero prefix
\(\tr_{\succ v}(a)=\tau\).
There are only finitely many possible configurations of:
\begin{equation}
 m_1=\lm(b),\quad s_1=\lm(L),\quad
 \lc(b),\quad\lc(L),\quad
 \tr_{\succ v/s_1}b,\quad\tr_{\succ v/m_1}L.
 \label{f-eq-3.1}
\end{equation}
Inside each configuration, the coefficient of \(a\) at \(v\) is
\begin{equation}
 c_1\beta+\beta_1c+c_0,
 \label{f-eq-3.2}
\end{equation}
where the fixed positive numbers \(\beta_1,c_1\) are the leading
coefficients of \(b,L\), while
\[
 \beta=\coef_{v/s_1}(b),\qquad
 c=\coef_{v/m_1}(L).
\]
The possible \(\beta\)'s are \(\N\)-like.
The possible \(c\)'s are \(\N\)-like if \(v/m_1\le1\), or if the
coefficient property of Proposition~\ref{f-prop:coefficients} is
already known at the smaller scale \(v/m_1\).
\end{lemma}
\begin{proof}
Write $M=\lm(\tau)$ and $t=\lc(\tau)>0$.
Since $\tau\ne0$ is the truncation above $v$, we have $M\succ v$.
The leading term of $bL$ is unchanged by this truncation, so
\[
 m_1s_1=M,\qquad \beta_1c_1=t,
 \quad\hbox{where }\beta_1=\lc(b)>0,
                         \ c_1=\lc(L)>0.
\]
The leading-monomial sets of $\Sk_{\ge x}$ and
$\{\log W:W\in\Sk,\ W>1\}$ are well ordered: both underlying
families are well ordered, and taking the leading monomial of a
positive series is weakly increasing.  If $m_1s_1=M$ admitted
infinitely many pairs, their first coordinates would be distinct;
an increasing sequence of them would force a strictly decreasing
sequence of second coordinates.  Hence only finitely many leading
monomial pairs occur.

Fix one such pair.  On a family with a fixed leading monomial,
taking the leading coefficient is weakly increasing.  The possible
$\beta_1$ and $c_1$ therefore belong to well-ordered subsets of
$\R_{>0}$.  An infinite set of pairs with $\beta_1c_1=t$ would
again have increasing first coordinates and decreasing second
coordinates.  Thus only finitely many leading-coefficient pairs
occur.  Fix one of them for the remainder of the configuration
argument.

Set $u=v/s_1$, $z=v/m_1$ and
\[
 B_+=\tr_{\succ u}b,\qquad L_+=\tr_{\succ z}L.
\]
Both are nonzero and positive, with leading terms
$\beta_1m_1$ and $c_1s_1$, because
$m_1/u=s_1/z=M/v\succ1$.  Their separate families are well
ordered, being weakly increasing truncation images of the label
family and logarithm family, respectively.  The remainder bounds
$b-B_+=O(u)$ and $L-L_+=O(z)$ give
\[
 bL-B_+L_+
 =(b-B_+)L_++B_+(L-L_+)+(b-B_+)(L-L_+)=O(v),
\]
because $us_1=zm_1=v$ and $uz=v^2/M\prec v$.
As $bL=\tau+O(v)$, it follows that
$B_+L_+=\tau+O(v)$.

For a fixed $B_+$, at most one $L_+$ can occur.  Indeed, a
nonzero difference of two such logarithmic prefixes is supported
strictly above $z$ and hence strictly dominates $z$.  Its product
with $B_+\arch m_1$ strictly dominates $v$, whereas subtracting
the two displayed product identities makes that product $O(v)$.

If infinitely many $B_+$ occurred, their well ordering would give
$B_0<B_1<\cdots$.  Write $L_i$ for the uniquely associated prefix
and $E_i=B_iL_i-\tau$, so $E_i=O(v)$.
For every fixed $i<j$, division gives the exact identity and bound
\begin{equation}
 \begin{split}
 L_i-L_j
 &=\frac{\tau(B_j-B_i)}{B_iB_j}
                  +\frac{E_i}{B_i}-\frac{E_j}{B_j}\\
 &=\frac{\tau(B_j-B_i)}{B_iB_j}+O(z)>0.
 \end{split}
 \label{f-eq-3.3}
\end{equation}
Here each error quotient is $O(v/m_1)=O(z)$, since
$B_i,B_j\arch m_1$.  The positive main term satisfies
\[
 \frac{\tau(B_j-B_i)}{B_iB_j}
 \arch\frac{s_1}{m_1}(B_j-B_i)
 \succ\frac{s_1}{m_1}u=z:
\]
$B_j-B_i>0$ is a nonzero difference of prefixes supported
strictly above $u$.  Thus the main term dominates the full error,
with no uniform choice of its real $O$-constant needed as $i,j$
vary.  Equation \eqref{f-eq-3.3} would make $L_0>L_1>\cdots$,
contradicting the well ordering of the prefix family.  There are
therefore finitely many pairs $(B_+,L_+)$, proving the asserted
finiteness of all configurations in \eqref{f-eq-3.1}.

Now fix a configuration.  The ambient monomials $u,z$ are valid
coefficient cutoffs even if $u\notin G$; in that case
$\coef_u(b)=0$.  Put
$\beta=\coef_u(b)$, $c=\coef_z(L)$ and write
\[
 b=B_++\beta u+r_b,\qquad
 L=L_++cz+r_L,\qquad r_b=o(u),\quad r_L=o(z).
\]
Since
$L_+=c_1s_1+o(s_1)$ and
$B_+=\beta_1m_1+o(m_1)$, multiplication yields
\[
 bL=B_+L_++(c_1\beta+\beta_1c)v+o(v).
\]
Indeed $\beta uL_+=c_1\beta v+o(v)$,
$czB_+=\beta_1cv+o(v)$, products involving $r_b$ or $r_L$
are $o(v)$, and the two boundary terms multiply to
$\beta c uz=o(v)$.  Taking the Hahn coefficient at $v$ gives
\eqref{f-eq-3.2} with the fixed real constant
$c_0=\coef_v(B_+L_+)$.  Its coefficient sum is finite by the
Hahn support argument above, even when either prefix is infinite.

The expansion $b=B_++\beta u+o(u)$ has a fixed prefix
$B_+\succ u$. Proposition~\ref{sf:prop:coef}, applied to the
actual Skolem labels in this configuration, therefore shows that
the possible $\beta$ form an $\N$-like set.
If $z\le1$, apply Proposition~\ref{sf:prop:logcoef} to the
expansions $\log W=L_++cz+o(z)$ with their fixed prefix $L_+$.
It gives the same conclusion for the possible $c$.
Both propositions use actual Skolem references and exponents in $\T$,
so no additional ambient-field assumption enters this step.
If $z>1$, then $z=v/m_1\in M_1$, since $m_1\in G$, and
$z\le v/x\prec v$.  The assumed coefficient property at this
smaller scale, applied with the fixed prefix $L_+$, gives the
remaining assertion.
\end{proof}

\begin{proposition}[Finite-scale coefficient induction]
\label{f-prop:coefficients}
Fix an integer \(r\ge1\).  For every \(v\in M_1\cap(1,x^r]\), the following
property holds:
\begin{quote}
if a family \(\calW\subseteq\Sk_{>1}\) has a fixed truncation
\(\tr_{\succ v}\log W\), then the coefficients
\(\coef_v(\log W)\), \(W\in\calW\), form an \(\N\)-like set.
\end{quote}
\end{proposition}
\begin{proof}
Induct through the finite ordered set in Lemma~\ref{f-lem:scales}.
For each \(W\), choose a dominant unrestricted frame \(D\) in a
component decomposition.  Then
\begin{equation}
 D\le W\le kD,\qquad \log W-\log D=O(1),
 \label{f-eq-3.6}
\end{equation}
where \(k\) may depend on \(W\).
Since \(v>1\), the logarithms of these frames have the same prefix
above \(v\), and their \(v\)-coefficient equals that of \(\log W\).

The fixed prefix of a frame logarithm has only finitely many
representations by nonzero truncated atoms.  Indeed, the truncated
atom family is a weakly increasing image of a well order, and hence
is well ordered; the positive-sum theorem of \cite[\S4]{BM} makes
its finite-sum family well ordered.  Infinitely many possible
summands in a representation of a fixed total would give increasing
summands and decreasing complementary sums.  Once the finitely many
summand values are fixed, their multiplicities are also bounded.
Thus there are finitely many cells and finitely many nonzero atom
slots in each cell.

For a nonzero slot apply Lemma~\ref{f-lem:atom}.
If \(z=v/m_1>1\), then \(m_1\ge x\) implies \(z\le v/x\prec v\).
Moreover \(z\in M_1\), because \(m_1\in G\).
It is therefore one of the earlier scales in
Lemma~\ref{f-lem:scales}, and the induction hypothesis applies.
The possible coefficient in each nonzero slot is consequently
\(\N\)-like.

It remains to include all atoms whose truncation above \(v\) is
zero.  Only atoms with leading monomial exactly \(v\) affect the
coefficient at \(v\).
For such an atom \(b\log A\), the bound
\(b\log A=O(v)\le O(x^r)\) and \(\log A\ge\log2\) imply
\(b=O(x^r)\).  Since \(x^r<2^x\), the basic identity
\(\Sk_{<2^x}=\N[x]\setminus\{0\}\) from \cite{BM} shows that
\(b\) is a polynomial, and
\(\lm(b)=x^j\) for one of the finitely many \(1\le j\le r\).
Its leading coefficient is a positive integer, while
\[
 \lm(\log A)=v/x^j\prec v.
\]
If this last scale is \(>1\), its leading coefficient is
\(\N\)-like by the induction hypothesis, applied to the family with
zero prefix above that scale.  If it is at most \(1\), it must be
\(1\), and \(A\) is a positive integer; the coefficient belongs to
\(\log\N_{\ge2}\).  Products with positive integer leading
coefficients and the finite union over \(j\) remain \(\N\)-like.
The aggregate coefficient of every finite sum of these small atoms
is also \(\N\)-like by the observation following \eqref{f-eq-2.4}.

Finally, \(\log P\) for \(P\in\N[x]\setminus\{0\}\) contributes
only its degree at the scale \(\log x\), and contributes zero at
the other scales \(>1\).  Finite sums over the slots, followed by
the finite union over cells, prove the induction step.
\end{proof}

\section{Shallow frame alphabets}
\begin{proposition}[Polynomial-width frame alphabet]\label{shallow}
For each fixed integer $r\ge1$, frames in a window $[Y,YR]$ with
$R\ge1$ and $\log R=O(x^r)$ have at most $\omega^{2r}$
Archimedean classes, uniformly in $Y,R$. In particular their
alphabet has type below $\omega^\omega$. The same conclusion
holds for their nonzero images modulo the additive subgroup $O(Y)$.
\end{proposition}
\begin{proof}
Fix an actual frame $D_0$ in a nonempty window. Every frame there
has $\log(D/D_0)=O(x^r)$, so its logarithm has a fixed prefix
above $x^r$. Its class modulo $\arch$ is determined by the
coefficients at the $2r$ scales above $1$ in
Lemma~\ref{f-lem:scales}.
Proceed through these scales in decreasing order. After the
preceding coefficients are fixed, the entire prefix above the
next scale is fixed, because there are no intervening support
monomials. Proposition~\ref{f-prop:coefficients} gives a set
of at most $\omega$ choices for the next coefficient. The possible
vectors therefore have lexicographic order type at most
$\omega^{2r}$, by induction on the finite number of remaining
coordinates. The class of $1$, if present, causes no exception:
adjoining a zero coefficient preserves the coefficient property.
Comparison at the first differing coefficient is precisely the
order of the Archimedean classes, including negative coefficients.

Frames $D=O(Y)$ vanish in the additive quotient. For $D\succ Y$,
the subgroup $O(Y)$ is negligible compared with $D$, so that
quotient preserves Archimedean equivalence and inequivalence.
\end{proof}

\section{Parameters for the local proof}\label{local-setup}
All the following scales belong to $\T$, as stipulated in the
conventions and justified by Proposition~\ref{sf:prop:parameters}.
For the local argument fix
\begin{equation}
 H=2^h,\quad T=2^H,\quad B=\Sk_{<H},\quad
 \ell=\log H,\quad \theta=\log\log x,\quad
 \kappa=C\theta^2,\qquad C\ge4,\quad x^2\preceq h.
 \label{local-cap}
\end{equation}
There is no polynomial upper bound on $\log h$ here. Suppose
\begin{equation}
 |B|\le\EE(a_2),\quad |\Sk_{<H^\kappa}|\le\EE(a'),\quad
 |\Lambda(T)|\le\omega^\lambda,\qquad a_2\ge\omega^3,\quad\lambda\ge1.
 \label{local-sizes}
\end{equation}
Define
\begin{equation}
 a_{\rm leaf}=\max(a_2,a')+1,\qquad
 a_0=\max((a_2+2)\hs\lambda,a_{\rm leaf})+1.
 \label{local-a0}
\end{equation}
The cap and leaf bounds are size inputs for a single height
step; the final induction will supply them at the preceding height.
No local fibre estimate is assumed in this construction.

\section{Coarse families and a finite tree}

\begin{definition}[Coarse node]\label{l-def:coarse}
A coarse node is a pair $\calB=(U,q)$ with
$U\in\Sk_{<T}$, $q\ge x$, and $q=O(\ell)$.  Put
\begin{equation}
 \sigma_{\calB}=q\theta,\qquad t_{\calB}=q\theta/x,\qquad
 \Fr(\calB)=\{D<T:\ D\text{ is a frame},\
                                  \log(D/U)=O(q)\}.
 \label{l-eq-4.1}
\end{equation}
It is a leaf if $\log U\le\ell\theta^2$.
\end{definition}

All the logarithms in $\Fr(\calB)$ have the same truncation above
$\sigma_{\calB}$.  Since $\log P=O(\log x)\prec\sigma_{\calB}$,
this common truncation is a sum of nonzero atom truncations.
Lemma~\ref{l-lem:representations} supplies finitely many cells.

\begin{lemma}[Coarse label classes]\label{l-lem:labels}
In a coarse slot with prescribed nonzero truncation $\tau$, the
labels fall into finitely many equivalence classes for
\begin{equation}
 b\sim_\tau b'\quad\Longleftrightarrow\quad
               b/b'=1+O(\sigma_{\calB}/\tau).
 \label{l-eq-4.2}
\end{equation}
Choose an actual atom $b_s\log W_s$ in each class.
If $\tau\succ b_s\sigma_{\calB}$, every label in that class
equals $b_s$ exactly.
Otherwise every base in that class is $<H^\kappa$.
Every coarse-zero atom has its base $<H$.
\end{lemma}
\begin{proof}
Write $s=\sigma_{\calB}$ and $\delta=s/\tau\prec1$.
The relation $b/b'=1+O(\delta)$ is an equivalence relation:
inversion and multiplication preserve $1+O(\delta)$ because
$\delta$ is infinitesimal. Its classes are convex in the
well-ordered label set. Indeed, if $b_1\le b\le b_2$ and
$b_2/b_1=1+O(\delta)$, then
$0\le b/b_1-1\le b_2/b_1-1=O(\delta)$.

Suppose there were infinitely many classes. Choose strictly
increasing representatives $b_i$ from distinct classes and actual
atoms $b_iL_i$, $L_i=\log W_i$, in the prescribed slot. Write
$b_iL_i=\tau+e_i$ with $e_i=O(s)$. For $i<j$, put
$d_{ij}=1-b_i/b_j>0$. Inequivalence implies
$b_j/b_i-1\succ\delta$, and hence $d_{ij}\succ\delta$:
if $b_j/b_i-1$ is infinitesimal this follows by division by
$b_j/b_i\sim1$; otherwise $d_{ij}$ is bounded below by a
positive real number. Thus
\begin{equation}
 \begin{aligned}
 L_i-L_j
 &=\frac{\tau}{b_i}d_{ij}
       +\frac{e_i}{b_i}-\frac{e_j}{b_j}\\
 &=\frac{s}{b_i}\left(\frac{d_{ij}}{\delta}+O(1)\right)>0.
 \end{aligned}
 \label{label-decrease}
\end{equation}
The constants in the remainder may depend on $i,j$, which are
fixed in each comparison. Since $d_{ij}/\delta$ is infinite,
the sign follows. This is a decreasing sequence in the
well-ordered family $\{\log W:W\in\Sk,\ W>1\}$, a contradiction.

In a fixed label class, $b/b_s=1+O(\delta)$, so
\[
 b-b_s=O(b_s\delta)=O(b_ss/\tau).
\]
If $\tau\succ b_ss$, this is $o(1)$.
Lemma~\ref{sf:lem:integers} makes the difference an integer;
as it is infinitesimal, it is zero. Thus $b=b_s$ as germs.

In the other case $\tau=O(b_ss)$ and $b\sim b_s$; consequently
\[
 \log W=\frac{\tau+O(s)}b=O(s)=O(\ell\theta)
                              <\kappa\ell
\]
eventually. Thus $W<H^\kappa$.
For a coarse-zero atom, $b\log W=O(s)$ and $b\ge x$, whence
\[
 \log W=O(s/x)=O(\ell\theta/x)=o(\ell),
\]
so $W<H$. These are bounds on each member germ; the diverging
factor $\theta$ absorbs every fixed real constant.
\end{proof}

Call the first kind of class \emph{exact}.
For every exact class attach the child
\begin{equation}
 \calB_s=(W_s,q_s),\qquad q_s=q+\log b_s.
 \label{l-eq-4.3}
\end{equation}
Only finitely many children are attached, including repetitions
from different cells if necessary.  Every base in the corresponding
class satisfies
\begin{equation}
 \log(W/W_s)=O(\sigma_{\calB}/b_s)
                  =O(t_{\calB_s}).
 \label{l-eq-4.4}
\end{equation}
At a leaf no children are needed.

\begin{lemma}[Termination and parameters]\label{l-lem:tree}
This construction yields a finite tree.  Every node satisfies the
conditions of Definition~\ref{l-def:coarse}.
\end{lemma}
\begin{proof}
Along a branch of length $j$,
\begin{equation}
 q_j=q_0+\sum_{i=1}^j\log b_i\le q_0+j\ell.
 \label{l-eq-4.5}
\end{equation}
For each finite $j$ this is $O(\ell)$ and is at least $x$.
No uniform real constant in $j$ is required by \eqref{l-eq-4.1}.

At a nonleaf, $q=o(\log U)$, since
$\log U>\ell\theta^2$ and $q=O(\ell)$.
An atom used to choose a child occurs in a frame $D\in\Fr(\calB)$,
so
\begin{equation}
 \log W_s\le\frac{\log D}{b_s}
       \le\frac{\log D}{x}\prec\log U .
 \label{l-eq-4.6}
\end{equation}
The positive classes $[\log W]$, $1<W\in\Sk$, are well ordered
as a monotone image of $\Sk$.  Thus there is no infinite branch.
Finite branching and K\"onig's lemma imply that the tree is finite.
\end{proof}

\begin{remark}
The essential invariant is \eqref{l-eq-4.1}, not a literal bound
$|\log(D/U)|\le q$.  The factor $\theta$ absorbs every real
constant when taking a common truncation.
Equation \eqref{l-eq-4.5} replaces any estimate on a product of labels.
\end{remark}

\section{The complete small-atom payload and frame coding}

Define the actual ordered set
\begin{equation}
 \Bsmall=\{0\}\cup
 \Sigma^+\{b\log W:\ b,W\in B,\ b\ge x,\ W>1\}.
 \label{l-eq-5.1}
\end{equation}
It is not a quotient or a truncation.

\begin{lemma}[Nonrecursive bounds]\label{l-lem:payload}
The order type of $\Bsmall$ is $<\EE(a_0)$.
The full atom family with label in $B$ and base in
$\Sk_{<H^\kappa}$ has type at most $\EE(a_{\rm leaf})<\EE(a_0)$.
The polynomial factors also have type $<\EE(a_0)$.
At a coarse leaf, all frames under consideration and all their
finite positive sums belong to $\Sk_{<H^\kappa}$.
\end{lemma}
\begin{proof}
The atom set in \eqref{l-eq-5.1} has order at most
$\EE(a_2)\od\EE(a_2)<\EE(a_2+1)$ and has at most
$\omega^\lambda$ Archimedean classes.  Equations \eqref{ext-sums}--\eqref{ext-ordinals} give
\begin{equation}
 |\Bsmall|
 \le\bigl(\EE(a_2+1)^\omega\bigr)^{\od\omega^\lambda}
 =\EE((a_2+2)+\lambda)
 \le\EE((a_2+2)\hs\lambda)<\EE(a_0).
 \label{l-eq-5.2}
\end{equation}
Adjoining zero is absorbed by this infinite limit bound.
The second assertion follows from the natural product of the two
size inputs.  Polynomial order type is $\omega^\omega$.

At a leaf, a frame in \eqref{l-eq-4.1} satisfies
\begin{equation}
 \log D\le\ell\theta^2+O(\ell)<2\ell\theta^2.
 \label{l-eq-5.3}
\end{equation}
Any particular finite sum of such frames is consequently
$<H^{4\theta^2}\le H^\kappa$.
The statement is uniform as a containment of germs:
each finite real coefficient is eventually absorbed by $\theta^2$.
\end{proof}

\begin{lemma}[Coding a fine frame family]\label{l-lem:coding}
Fix a coarse nonleaf $\calB$ and a family
$\calA\subseteq\Fr(\calB)$ whose logarithms have a common truncation
above $\sigma$, where $0<\sigma\preceq\sigma_{\calB}$.
Fix $\eta>0$. To count $\log\calA/O(\eta)$, it suffices to count finitely many
products of the following ordered coordinates:
\begin{enumerate}[label=\textup{(\roman*)},leftmargin=3em]
\item a polynomial factor, an element of $\Bsmall$, and finitely
many atoms with label in $B$ and base below $H^\kappa$;
\item for each remaining slot, a base band
$\log(W/W_f)=O(\sigma/b_s)$, counted modulo
$\equiv_{\eta/b_s}$, at one of the children $\calB_s$.
\end{enumerate}
Here $W_f$ is an actual base and satisfies
$\log(W_f/W_s)=O(t_{\calB_s})$.
If every primitive recursive coordinate used in \textup{(ii)}
has order at most $\EE(d)$, with $d\ge a_0$, finite products and
finite unions in this coding have order at most $\EE(d+2)$.
Each coordinate in \textup{(ii)} may itself be bounded by a product of
two such primitive bounds without changing this conclusion.
\end{lemma}
\begin{proof}
Write $c=\sigma_{\calB}$ and
$\pi_u=\tr_{\succ u}$.  Since $\sigma\preceq c$,
$\pi_c\pi_\sigma=\pi_c$.  Equal $\pi_u$-truncations differ by
$O(u)$.  Apply Lemma~\ref{l-lem:representations} to the common
fine truncation, using both atom truncations and the nonzero
truncations of $\log P$ as summands.  This gives finitely many
fine cells.  The coarse cells are those already used in constructing
$\calB$; the polynomial logarithm has zero coarse truncation.

Here is an explicit common refinement of these two finite lists.
Number the occurrences in a coarse cell as
$(\tau_1,\ldots,\tau_r)$ and those in a fine cell as
$(\nu_1,\ldots,\nu_m)$, even when some entries coincide.
Retain every injection
\[
 \iota:\{1,\ldots,r\}\longrightarrow\{1,\ldots,m\}
\]
such that
\[
 \pi_c(\nu_{\iota(i)})=\tau_i\quad(1\le i\le r),
 \qquad
 \pi_c(\nu_j)=0\quad(j\notin\operatorname{im}\iota).
\]
For each coarse occurrence also specify one of its finitely many
label classes from Lemma~\ref{l-lem:labels}.  Also record a
polynomial tag: either one unmatched fine occurrence representing
the nonzero fine truncation of $\log P$, or the symbol $0$ when
that truncation is zero.  There are finitely
many choices of all this data.  Discard choices realized by no
frame decomposition in $\calA$, and denote the remaining finite
index set by $\Delta$.

Every decomposition of every $D\in\calA$ realizes at least one
index in $\Delta$: match the coarse and fine truncations of each
actual coarse-nonzero atom occurrence, and then choose the recorded
label classes.  Repeated occurrences are matched individually.
No uniqueness of a decomposition or of a matching is asserted or
needed; all compatible choices were retained.

Fix $\delta\in\Delta$.  Let $I_\delta$ and $J_\delta$ be its
exact and nonexact coarse occurrences, respectively.  In every
decomposition realizing $\delta$, an occurrence $i\in I_\delta$
has the fixed label $b_{s(i)}$ and the prescribed fine truncation
$\nu_{\iota(i)}$.  Choose one actual decomposition realizing
$\delta$ and let $W_{f,i}$ be its base at that occurrence.
For every other base $W$ occurring there,
\[
 b_{s(i)}\log(W/W_{f,i})=O(\sigma),
 \qquad
 \log(W/W_{f,i})=O(\sigma/b_{s(i)}).
\]
The reference base belongs to the chosen coarse label class, so
\eqref{l-eq-4.4} gives
\[
 \log(W_{f,i}/W_{s(i)})=O(t_{\calB_{s(i)}}).
\]
Use the whole band
\[
 X_{\delta,i}=
 \{W\in\Sk_{<T}:\log(W/W_{f,i})=O(\sigma/b_{s(i)})\}
\]
as the coordinate before quotienting.  Its radius is
$O(t_{\calB_{s(i)}})$, because
$\sigma/b_{s(i)}=O(q\theta/x)$ and $q_{s(i)}\ge q$.
Thus the band has the required child and reference parameters.

Every coarse-zero atom has both its label and its base in $B$ by
Lemma~\ref{l-lem:labels}.  Record the sum of all such atoms,
including any that are nonzero at the fine cutoff, as one actual
element of $\Bsmall$; use zero if there are none.  Record $P$
separately.  At every $j\in J_\delta$, the base is below
$H^\kappa$, so record the entire actual atom as an element of
\[
 \mathcal A_{\rm sm}=
 \{b\log W:b\in B,\ b\ge x,\ W\in\Sk_{<H^\kappa},\ W>1\}.
\]
There are only $|J_\delta|$ such coordinates.  In particular, an
unbounded number of coarse-zero occurrences creates no additional
coordinates.

We now construct the ordered map which performs the counting.
Let $\mathcal H=K+K\log x$, regarded as an ordered additive group,
and let $J_\eta=\{h\in\mathcal H:h=O(\eta)\}$, a convex
subgroup.  All logarithmic coordinates above belong to
$\mathcal H$.  Put $\mathcal P=\N[x]\setminus\{0\}$ and form
the full Cartesian product of well orders
\[
 \mathcal D_\delta=
 \mathcal P\times\Bsmall\times
 \prod_{j\in J_\delta}\mathcal A_{\rm sm}\times
 \prod_{i\in I_\delta}
       (X_{\delta,i}/{\equiv_{\eta/b_{s(i)}}}).
\]
Empty products have their usual singleton interpretation.
Define a map to the ordered quotient group by
\[
 \Phi_\delta\bigl(P,A,(a_j)_j,([W_i])_i\bigr)
 =\left[\log P+A+\sum_{j\in J_\delta}a_j
                   +\sum_{i\in I_\delta}b_{s(i)}\log W_i\right]_{J_\eta}.
\]
This is well defined: changing a representative $W_i$ changes its
weighted logarithm by $O(\eta)$, and there are only finitely many
such terms.  It is weakly increasing in each coordinate.
For a quotient coordinate this follows directly from the convex
quotient orders and $b_{s(i)}>0$: the map
$[W]\mapsto[b_{s(i)}\log W]_{J_\eta}$ is increasing.
For the other coordinates it follows from monotonicity of
$P\mapsto\log P$, addition, and the quotient projection.
This argument requires no choice of representatives compatible
with the product order.

For any actual decomposition realizing $\delta$, inserting its
polynomial, its complete coarse-zero sum, its nonexact atoms, and
its exact-base classes in $\mathcal D_\delta$ gives precisely
$[\log D]_{J_\eta}$.  Consequently
\[
 \log\calA/O(\eta)
 \ \subseteq\ \bigcup_{\delta\in\Delta}
                         \Phi_\delta(\mathcal D_\delta)
\]
as ordered subsets of $\mathcal H/J_\eta$.
The full products may produce additional values; this containment
is exactly what an upper bound requires.  By (BM2), each image
has the natural-product bound of its coordinate order types, and
their finite union has the natural-sum bound.  This proves the
claimed reduction without assigning a unique code to a frame.

Finally, Lemma~\ref{l-lem:payload} bounds each coordinate of
type \textup{(i)} by $\EE(a_0)\le\EE(d)$.  If a coordinate
of type \textup{(ii)} has a bound which is a natural product of
two primitive bounds, expand that bound into two factors when
forming the natural product.  There are still only finitely many
factors in each of finitely many products.  For every fixed
positive integer $N$,
\[
 \EE(d)^{\od N}=\omega^{\omega^d N}<\EE(d+1).
\]
A finite natural sum of these bounds is still below
$\EE(d+1)$, and in particular is at most $\EE(d+2)$, as claimed.
\end{proof}

\begin{lemma}[Actual quotient maps and assembly]\label{l-lem:quotient-maps}
Write $q_Y$ for the projection of the ambient ordered additive
group modulo $O(Y)$.  The following rules will be used on actual
well-ordered families of positive values.
\begin{enumerate}[label=\textup{(\roman*)},leftmargin=3em]
\item If a convex equivalence relation $\sim_f$ refines another
convex relation $\sim_c$ on a well-ordered family $A$, then
\[
 A/{\sim_f}\longrightarrow A/{\sim_c},\qquad
 [a]_f\longmapsto[a]_c
\]
is a weakly increasing surjection.  If the coarse quotient has
type at most $\beta$ and each fine fibre has type at most $\alpha$,
then
\[
 |A/{\sim_f}|\le\alpha\beta\le\alpha\od\beta.
\]
The same bounds hold after restricting $A$ to any subfamily.

\item Suppose $0<\eta\preceq1$ and every $D\in\mathcal D$
satisfies $D=O(Z)$.  There is a weakly increasing surjection
\begin{equation}
 (\log\mathcal D)/O(\eta)\longrightarrow q_{\eta Z}(\mathcal D),
 \qquad [\log D]_{O(\eta)}\longmapsto q_{\eta Z}(D).
 \label{l-eq:log-to-additive}
\end{equation}
In particular, a bound for the actual logarithmic quotient
bounds the indicated additive image.

\item Suppose $\varepsilon\prec1$ and
$Y=O(\varepsilon W)$ for every $W\in\mathcal M$.  There is a
weakly increasing surjection
\begin{equation}
 q_Y(\mathcal M)\longrightarrow\mathcal M/{\equiv_\varepsilon},
 \qquad q_Y(W)\longmapsto[W]_{\equiv_\varepsilon}.
 \label{l-eq:additive-to-log}
\end{equation}
When $Y=\varepsilon V$ and every $W\in\mathcal M$ is
Archimedean-equivalent to $V$, the two relations on $\mathcal M$
coincide.

\item For a positive family $\mathcal D$, let
$P_Y=q_Y(\mathcal D)\setminus\{0\}$.  Then $P_Y$ is positive,
\begin{equation}
 |P_Y/\arch|\le|\mathcal D/\arch|,\qquad
 q_Y(\Sigma^0\mathcal D)=\Sigma^0P_Y.
 \label{l-eq:positive-quotient-alphabet}
\end{equation}
For any fixed finite list of well-ordered subsets $X_i$ of the
additive quotient, the addition map
\begin{equation}
 \prod_{i=1}^kX_i\longrightarrow X_1+\cdots+X_k,
 \qquad (u_1,\ldots,u_k)\longmapsto\sum_{i=1}^ku_i
 \label{l-eq:finite-quotient-sum}
\end{equation}
is increasing in every coordinate, and its image has type at
most $\bigodot_{i=1}^k|X_i|$.  A finite union of such images is
bounded by the natural sum of their bounds.
\end{enumerate}
\end{lemma}
\begin{proof}
For (i), the fine quotient is the ordered sum of its convex
fibres over the coarse quotient.  The displayed bounds follow
from the ordered-sum estimate and (BM2).  Convex quotient maps
and inclusions of ordered subfamilies preserve the relevant
order bounds.

For (ii), if $\log(D/D')=O(\eta)$, then
$D/D'-1=O(\eta)$ because $\eta=O(1)$.  Therefore
$D-D'=O(\eta D')=O(\eta Z)$.  This proves well-definedness;
monotonicity follows from the monotonicity of logarithm,
exponential and the convex quotient projections.  For (iii),
$W-W'=O(Y)$ implies $(W-W')/W=O(\varepsilon)$ and hence
$\log(W/W')=O(\varepsilon)$.  If $W,W'\arch V$, the converse
follows from $W-W'=O(\varepsilon W')=O(\varepsilon V)$.

For (iv), a positive value has nonzero image precisely when it
strictly dominates $Y$.  For $D,D'\succ Y$, Archimedean
equivalence of $q_Y(D)$ and $q_Y(D')$ is equivalent to that of
$D$ and $D'$.  Indeed an inequality
$q_Y(D)\le n q_Y(D')$ gives
$D\le nD'+O(Y)=O(D')$, and conversely an inequality in the
original group passes to the quotient.  The ordered alphabet
of the positive images is thus the alphabet of the subfamily
$\{D\in\mathcal D:D\succ Y\}$.  The equality for finite sums
follows from additivity of $q_Y$.  Finally, addition is increasing
in each coordinate, so (BM2) proves the finite-product and
finite-union assertions.
\end{proof}

\begin{remark}[Restriction before exponentiation]\label{l-rem:restricted-exp}
The containing image in Lemma~\ref{l-lem:coding} need not consist
of logarithms of frames in the original family.  Thus, when a
bound such as $D=O(Z)$ is needed, first restrict that containing
image to the actual quotient $(\log\mathcal D)/O(\eta)$.
Choose in each of its classes the least representative belonging
to $\log\mathcal D$, exponentiate that representative, and then
apply $q_{\eta Z}$.  This gives exactly the map
\eqref{l-eq:log-to-additive}; every representative satisfies the
required value bound.  Values produced by arbitrary combinations
of coding coordinates are used only to bound the containing
ordered image.  Likewise, after adding quotient coordinates,
restrict to the actual member quotient before applying
\eqref{l-eq:additive-to-log}.
\end{remark}

\section[Local families at a prescribed width]{Local families at a prescribed window width}\label{l-ordinary-defs}

Fix a positive width scale $L\succeq x$. For this section suppose
one ordinal $\omega^s$, $s>0$, bounds all frame alphabets in windows
with $\log R=O(L)$, uniformly in their centers and ratios. We will
prove these alphabet bounds successively in Section~\ref{window-tiers}.
We now specify exactly which local nodes the coarse induction controls.
All frames, all members of the centered families, and all members
of the Archimedean classes are restricted to values $<T$.
The auxiliary finite sums defining $\SSS$ are unrestricted;
at leaves they lie below $H^\kappa$ by Lemma~\ref{l-lem:payload}.
Fix $\calB=(U,q)$.
Write
\[
 \calL(V;w)=\{W\in\Sk_{<T}:|\log(W/V)|=O(w)\}.
\]
The counts below, together with their parameter restrictions, define
the admissible nodes at $\calB$.

\begin{enumerate}[label=\textup{(\Alph*)},leftmargin=3em]
\item $\NN(V;\varepsilon,R)=Q_\varepsilon(\calL(V;\varepsilon R))$,
where
\begin{equation}
 \varepsilon\prec1,\quad R\ge1,\quad\varepsilon R\prec1,\quad
 \log R=O(L),\quad \log(1/\varepsilon)=O(q),\quad
 |\log(V/U)|=O(t_{\calB}).
 \label{l-eq-6.1}
\end{equation}
\item $\SSS(Y;R)$ is the order type, modulo additive $O(Y)$,
of all finite nonnegative sums of frames in $[Y,YR]$, where
\begin{equation}
 R\ge1,\quad \log R=O(L),\quad
 |\log(Y/U)|=O(q),\quad \log R=O(q).
 \label{l-eq-6.2}
\end{equation}
\item $\KK(E;t)=Q_{e^{-t}}(E)$, where $E$ is an Archimedean class,
has a representative $V$ with $|\log(V/U)|=O(t_{\calB})$, and
\begin{equation}
 \log x\le t,\qquad t=O(L),\qquad t=O(q).
 \label{l-eq-6.3}
\end{equation}
\item $\FF(V;w)=Q_1(\calL(V;w))$, where
\begin{equation}
 w>0,\quad w=O(t_{\calB}),\quad \log(1+w)=O(L),\quad
 |\log(V/U)|=O(t_{\calB}).
 \label{l-eq-6.4}
\end{equation}
\item $\GG(V;w,p)=Q_{e^{-p}}(\calL(V;w))$, with the radius and center conditions of \eqref{l-eq-6.4}, but without
its $\log(1+w)$ restriction, and
\begin{equation}
 p\ge\log x,\qquad p=O(q).
 \label{l-eq-6.5}
\end{equation}
\end{enumerate}
The reference $V$ in a centered node belongs to $\Sk_{<T}$.
The scale $Y$ belongs to $\T_{>0}$ and need not be a Skolem function.
All subsequent changes of scale remain in $\T$ by
Proposition~\ref{sf:prop:parameters}.
If an intersection used during the proof is a subfamily of one of
these nodes, the full node is used as an upper bound.

At a coarse leaf every node in this list is bounded by $\EE(a_0)$.
For (B), use \eqref{l-eq-5.3} and the full finite-sum containment.
For the other nodes a member satisfies
$\log W\le\log U+O(t_{\calB})+O(1)$ and therefore
$W<H^\kappa$.  Since $t_{\calB}=O(\ell\theta/x)$,
the same fixed leaf fragment works for all five types.

\subsection{A common check for recursive base nodes}

Whenever Lemma~\ref{l-lem:coding} is used, a recursive slot has
label $b=b_s$ and fine base center $W_f$ satisfying \eqref{l-eq-4.4}.
Its radius is a constant multiple, in the $O$ sense, of $\sigma/b$.
Thus
\begin{equation}
 \frac{\sigma}{b}=O(q\theta/x)=O(t_{\calB_s}).
 \label{l-eq-6.6}
\end{equation}
All recursive centers and their member bands therefore satisfy the
child location restrictions.  The child depth allowance is
$q_s=q+\log b$.
We verify the other restrictions separately in the five operations
below.  In particular no literal radius inequality follows merely
from \eqref{l-eq-6.6}, and none is required.

\subsection{The fixed-ratio operation \texorpdfstring{$\NN$}{N}}\label{l-sec:N}

Put $S=\varepsilon R V$.
The members have a common additive truncation above $S$.
Take its finite cells of large frames, with representative $D_m$
in each slot.  Then
\begin{equation}
 D-D_m=O(S),\qquad
 s_m=S/D_m\prec1,\qquad D_m\le2V.
 \label{l-eq-6.7}
\end{equation}
The last inequality follows because the member is asymptotic to $V$.
Use Lemma~\ref{l-lem:coding} on this lift family with
\begin{equation}
 \sigma=\theta s_m,\qquad \eta_m=\varepsilon V/D_m.
 \label{l-eq-6.8}
\end{equation}
Here $\sigma\prec\theta\prec\sigma_{\calB}$.
For each exact label $b$, its base quotient is bounded by the child
\begin{equation}
 \NN(W_f;\eta_m/b,\theta R).
 \label{l-eq-6.9}
\end{equation}
The width is $\theta s_m/b\prec\theta/x\prec1$ and
$\log(\theta R)=O(L)$.
Its depth obeys
\begin{equation}
 \log(b/\eta_m)
   \le \log b+\log(1/\varepsilon)+O(1)=O(q+\log b).
 \label{l-eq-6.10}
\end{equation}
All lifted frames lie in $\Fr(\calB)$: from $D_m\succ S$ and
$D_m\le2V$, their logarithmic distance from $V$ is
$O(\log(1/\varepsilon))=O(q)$.
The logarithmic precision $\eta_m$ makes differences of the frames
$O(\varepsilon V)$.

The remaining frames satisfy $D=O(S)$.
The upper endpoint must include every such frame, including $D>S$.
Set
\begin{equation}
 S^\#=\min(xS,V),\qquad
 R^\#=\frac{S^\#}{\varepsilon V}
                  =\min(xR,1/\varepsilon).
 \label{l-eq-6.11}
\end{equation}
Since $S\prec V$, every $D=O(S)$ satisfies $D<S^\#$ eventually.
Those $D<\varepsilon V$ disappear modulo $O(\varepsilon V)$;
all others belong to the same-coarse node
$\SSS(\varepsilon V;R^\#)$.
It is admissible because
\[
 \log R^\#\le\log R+\log x=O(L),\quad
 \log R^\#\le\log(1/\varepsilon)=O(q),\quad
 \log(\varepsilon V/U)=O(q).
\]
For clarity, put $Y_0=\varepsilon V$ and
$\mathcal M=\calL(V;\varepsilon R)$.  Index the finitely many
additive cells by $c\in\mathcal C$, and let
$\mathcal D_{c,m}$ be the actual lift family in its $m$-th
large-frame slot.  Put
\[
 X_{c,m}=q_{Y_0}(\mathcal D_{c,m}),\qquad
 X_0=q_{Y_0}\bigl(\Sigma^0\{D:\ D\text{ is a frame},\
                                      Y_0\le D\le Y_0R^\#\}\bigr).
\]
For $D\in\mathcal D_{c,m}$ one has $D/D_m=1+o(1)$ and
$\eta_mD_m=Y_0$.  Lemma~\ref{l-lem:quotient-maps}(ii), applied
after the restriction in Remark~\ref{l-rem:restricted-exp}, gives
\[
 |X_{c,m}|\le |(\log\mathcal D_{c,m})/O(\eta_m)|.
\]
Also $|X_0|=\SSS(Y_0;R^\#)$.  In the quotient, every member of
$\mathcal M$ is the sum of one coordinate from each large-frame
slot of some cell and one coordinate from $X_0$; frames below
$Y_0$ have zero image.  Consequently
\begin{equation}
 q_{Y_0}(\mathcal M)\subseteq
 \bigcup_{c\in\mathcal C}
 \bigl(X_{c,1}+\cdots+X_{c,k_c}+X_0\bigr),\qquad
 \NN(V;\varepsilon,R)
 \le\bigoplus_{c\in\mathcal C}
       \left(|X_0|\od\bigodot_{m=1}^{k_c}|X_{c,m}|\right).
 \label{l-eq:N-assembly}
\end{equation}
Here an empty list of large slots contributes the singleton
$\{0\}$ and an empty natural product is $1$.
The inequality follows from the explicit addition maps
\eqref{l-eq:finite-quotient-sum}, the finite-union bound, and
Lemma~\ref{l-lem:quotient-maps}(iii), since every member of
$\mathcal M$ is asymptotic to $V$.  The containing sums need not
belong to $\mathcal M$: the final map to
$\mathcal M/{\equiv_\varepsilon}$ is used only on the actual
subfamily $q_{Y_0}(\mathcal M)$.  This proves the $\NN$
operation, including its boundary region.

\subsection{The shallow positive-sum operation \texorpdfstring{$\SSS$}{S}}\label{l-sec:S}

If $R=O(1)$, all sums are $O(Y)$ and the count is one.
Otherwise put
\begin{equation}
 \sigma=\theta\log R,\qquad \eta=1/R.
 \label{l-eq-6.12}
\end{equation}
The input frames belong to $\Fr(\calB)$ and have a common logarithmic
truncation above $\sigma\preceq\sigma_{\calB}$.
Counting their logarithms modulo $O(1/R)$ bounds their images modulo
$O(Y)$, since they are at most $YR$.

In an exact slot split according to the fixed comparison of $b$
and $\sigma$.
If $b\succ\sigma$, use
\begin{equation}
 \NN(W_f;1/(Rb),R\sigma).
 \label{l-eq-6.13}
\end{equation}
Its width $\sigma/b$ is infinitesimal; its logarithmic ratio is
$\log R+\log\sigma=O(L)$ and its depth is
$\log R+\log b=O(q+\log b)$.
If $b\preceq\sigma$, use
\begin{equation}
 \FF(W_f;\sigma/b)\ \od\
       \sup_E\KK(E;\log(Rb)).
 \label{l-eq-6.14}
\end{equation}
The supremum ranges over the Archimedean classes meeting that base
band.  Each has an actual representative within it and so is located
in the child by \eqref{l-eq-6.6}.
The moving child in \eqref{l-eq-6.14} also satisfies
$\log(1+\sigma/b)=O(\log\sigma)=O(L)$.
Here $b=O(\sigma)$ gives $\log b=O(\log\sigma)=O(L)$, making
the $\KK$ node admissible; $\log(Rb)=O(q+\log b)$ and
$\log(Rb)\ge\log x$.
To justify \eqref{l-eq-6.14} explicitly, let
$\mathcal W=\calL(W_f;\sigma/b)$ and
$\delta=1/(Rb)$.  Since $\delta\prec1$, the map
\[
 \mathcal W/{\equiv_\delta}\longrightarrow
 \mathcal W/{\arch},\qquad [W]_{\equiv_\delta}\longmapsto[W]
\]
is a convex quotient map.  Its index has type at most
$\FF(W_f;\sigma/b)$, and its fibre over any class $E$ has type
at most $\KK(E;\log(Rb))$.  Lemma~\ref{l-lem:quotient-maps}(i)
gives exactly the natural product in \eqref{l-eq-6.14}; using
the whole admissible class $E$ only enlarges its fibre.

Let $\mathcal D=\{D:\ D\text{ is a frame},\ Y\le D\le YR\}$,
and let $A\ge2$ bound the actual quotient
$(\log\mathcal D)/O(1/R)$ furnished by the coding lemma.
With $Z=YR$, the map
\[
 (\log\mathcal D)/O(1/R)\longrightarrow q_Y(\mathcal D),
 \qquad [\log D]\longmapsto q_Y(D)
\]
is justified by Lemma~\ref{l-lem:quotient-maps}(ii), with
representatives chosen in $\mathcal D$.  Thus the positive
set $P_Y=q_Y(\mathcal D)\setminus\{0\}$ has type at most $A$.
Its Archimedean alphabet has type at most $\omega^s$ by the
width-$L$ hypothesis and \eqref{l-eq:positive-quotient-alphabet}.
Additivity now identifies the required quotient with
$\Sigma^0P_Y$.  Applying (BM3) gives
\begin{equation}
 \SSS(Y;R)=|\Sigma^0P_Y|
       \le(A^\omega)^{\od\omega^s}.
 \label{l-eq-6.15}
\end{equation}
The empty and singleton cases are immediate; otherwise adjoining
the least element $0$ is absorbed by the displayed infinite
limit majorant.  Every recursive node in
\eqref{l-eq-6.13}--\eqref{l-eq-6.14} is at a coarse child.

\subsection{The Archimedean-fibre operation \texorpdfstring{$\KK$}{K}}\label{l-sec:K}

Fix a representative $V\in E$ and put $\varepsilon=e^{-t}$.
In every decomposition of a member of $E$, separate the frames
$D\arch V$ from those $D\prec V$.
The top frames have common logarithmic prefix above $1$.
Apply Lemma~\ref{l-lem:coding} with $\sigma=1$, $\eta=\varepsilon$.
The recursive base nodes are
\begin{equation}
 \NN(W_f;\varepsilon/b,e^t).
 \label{l-eq-6.16}
\end{equation}
Their widths are $1/b$, their logarithmic ratios are $t=O(L)$,
and their depths are $t+\log b=O(q+\log b)$.
Since $D\arch V$, this coding also counts their additive images
modulo $O(\varepsilon V)$.
All those positive images have one Archimedean class.
If their count is at most $A\ge2$, \eqref{ext-sums} bounds all their finite sums
by $A^\omega$.

For each smaller frame, \eqref{l-eq-2.4} gives $D=O(V/x)<V/\theta$.
Discarding $D<\varepsilon V$, the other small frames belong to
the same-coarse node
\begin{equation}
 \SSS(\varepsilon V;e^t/\theta).
 \label{l-eq-6.17}
\end{equation}
Its logarithmic ratio is $t-\log\theta=O(L)$ and $O(q)$;
it is at least $1$ because $t\ge\log x$.
The center condition is $\log(\varepsilon V/U)=O(q)$.
Put $Y_0=\varepsilon V$ and let
$\mathcal D_{\rm top}$ be the actual top-frame family.
Its logarithmic coding quotient maps increasingly onto
$P_{\rm top}=q_{Y_0}(\mathcal D_{\rm top})$, by
Lemma~\ref{l-lem:quotient-maps}(ii) with $Z=V$.
The memberwise bound $D=O(V)$ suffices for this map; its real
constant need not be uniform over the family.  All these images
are positive and have one Archimedean class.  Thus, if $A\ge2$ bounds the
logarithmic frame count, put
\[
 X_{\rm top}=\Sigma^+P_{\rm top},\qquad
 X_{\rm small}=q_{Y_0}\bigl(\Sigma^0\{D:\ D\text{ is a frame},\
                                  Y_0\le D\le V/\theta\}\bigr).
\]
Then $|X_{\rm top}|\le A^\omega$ by (BM3), while
$|X_{\rm small}|=\SSS(\varepsilon V;e^t/\theta)$.
Every member of $E$ has at least one top frame, and its quotient
lies in $X_{\rm top}+X_{\rm small}$.  Therefore the addition map
and the restriction to the actual member quotient give
\begin{equation}
 \KK(E;t)=|q_{Y_0}(E)|
 \le |X_{\rm top}|\od|X_{\rm small}|
 \le A^\omega\od\SSS(\varepsilon V;e^t/\theta).
 \label{l-eq:K-assembly}
\end{equation}
The first equality is Lemma~\ref{l-lem:quotient-maps}(iii): on
$E$, additive equivalence modulo $O(\varepsilon V)$ coincides
with $\equiv_\varepsilon$.  No sum outside $E$ is used as a
representative for that multiplicative quotient.

\subsection{The moving Archimedean-band operation \texorpdfstring{$\FF$}{F}}\label{l-sec:F}

For $w=O(1)$ the count is one.
Otherwise set
\begin{equation}
 \sigma=\theta w,\qquad \eta=1.
 \label{l-eq-6.18}
\end{equation}
A dominant frame of a member has logarithmic distance $O(w)$
from $V$ and hence a common prefix above $\sigma$.
These frames lie in $\Fr(\calB)$.
Moreover $\sigma=O(q\theta^2/x)\preceq\sigma_{\calB}$.
Every member has the Archimedean class of a dominant frame.
Frame coding at precision $1$ uses, for $b\succ\sigma$,
\begin{equation}
 \NN(W_f;1/b,\sigma),
 \label{l-eq-6.19}
\end{equation}
and, for $b\preceq\sigma$,
\begin{equation}
 \FF(W_f;\sigma/b)\ \od\ \sup_E\KK(E;\log b).
 \label{l-eq-6.20}
\end{equation}
In \eqref{l-eq-6.19}, $\log\sigma=\log w+\log\theta=O(L)$
by the added restriction $\log(1+w)=O(L)$.
In \eqref{l-eq-6.20}, $\log b=O(\log\sigma)=O(L)$.
Its new moving radius obeys $\sigma/b\le\theta w/x<w$ eventually,
so $\log(1+\sigma/b)=O(L)$ as required.
The depth and location conditions follow from \eqref{l-eq-6.6} and
$\log b\le q_s$.  All recursive nodes are at coarse children.

For \eqref{l-eq-6.20}, project the base-band quotient at precision
$1/b$ to its Archimedean quotient.  The index count is
$\FF(W_f;\sigma/b)$ and each convex fine fibre is bounded by
$\KK(E;\log b)$, so Lemma~\ref{l-lem:quotient-maps}(i)
gives the displayed product.  Let $\mathcal D_{\rm dom}$ be
the family of actual frames that occur as dominant frames of
members of $\calL(V;w)$.  For a decomposition with $k$ frames,
its dominant frame $D$ satisfies $D\le W\le kD$, hence
$[W]=[D]$.  Accordingly there is an inclusion of ordered class
sets
\[
 \calL(V;w)/\arch\ \subseteq\ \mathcal D_{\rm dom}/\arch.
\]
Frame coding bounds the right side, which is exactly
$(\log\mathcal D_{\rm dom})/O(1)$.  This is an inclusion of
class sets, so no monotonicity of a choice of dominant frame is
required.  In particular, when child primitive counts are at most
$\EE(d)$, $d\ge a_0$, the resulting $\FF$ bound is
$\EE(d+2)$.

\begin{lemma}[One ordinary coarse step]\label{ordinary-step}
Under the width-$L$ alphabet bound $\omega^s$, let
$c=s\hs\omega\hs10$. If all four ordinary primitive counts
at each coarse child are at most $\EE(d)$, $d\ge a_0$, then
all four counts at the parent are at most $\EE(d\hs c)$.
\end{lemma}
\begin{proof}
Fine coding gives $\EE(d+2)$ for each frame family.
The positive-sum operation has bound
\begin{equation}
 \bigl(\EE(d+2)^\omega\bigr)^{\od\omega^s}
 =\EE((d+3)+s).
 \label{ordinary-cost}
\end{equation}
The moving Archimedean count already has the bound $\EE(d+2)$.
Set $v=(d+3)+s$.  The same-node positive-sum count is at most
$\EE(v)$, and every large-frame coordinate in
\eqref{l-eq:N-assembly} is at most $\EE(d+2)\le\EE(v)$.
For $\KK$, equation \eqref{l-eq:K-assembly} combines a
top-frame sum of type at most
$\EE(d+2)^\omega=\EE(d+3)$ with a same-node positive-sum
count at most $\EE(v)$.  These are finite products and, for
$\NN$, a finite union of such products.  For a fixed positive
integer $k$,
\[
 \EE(v)^{\od k}=\omega^{\omega^v\cdot k};
\]
any fixed finite natural sum of these terms is less than
$\omega^{\omega^{v+1}}=\EE(v+1)$, and in particular at most
$\EE(v+2)$.  Thus all four operations are bounded by
\[
 \EE(((d+3)+s)+2)\le\EE(d\hs c),
 \qquad ((d+3)+s)+2\le d\hs s\hs5<d\hs c.
\]
Compute $\SSS,\FF$ first, then $\NN,\KK$; the first pair
uses proper coarse children only.
\end{proof}

\section{Window scales at a fixed tower cap}\label{window-tiers}

Fix \(n\ge2\) and a fixed integer \(m\ge2\), and set
\begin{equation}
 H=E_{n-1}(x^m),\quad T=E_n(x^m),\quad
 J=n-2,\qquad L_j=E_j(x)\quad(0\le j\le J).
 \label{h-eq-3.1}
\end{equation}
The window scales deliberately use \(x\), while the cap uses
\(x^m\).

\begin{lemma}\label{h-lem:scales}
These parameters satisfy
\begin{equation}
 x=O(\ell),\qquad \log\ell=O(L_J),\qquad
 xL_j=O(\ell)\quad(0\le j\le J).
 \label{h-eq-3.2}
\end{equation}
Also \(L_0=x\) and \(\log L_{j+1}=(\log2)L_j\).
\end{lemma}
\begin{proof}
For \(n=2\), \(\ell=(\log2)x^m\) and \(J=0\), so the
claims follow from \(m\ge2\).
For \(n\ge3\), \(\ell=(\log2)E_{n-2}(x^m)\).
Since \(x^m<2^x\) eventually,
\[
 E_{n-3}(x^m)<E_{n-2}(x)=L_J,
\]
which proves the logarithmic bound. For every fixed \(k\ge1\),
\(E_k(x^m)/E_k(x)\succ x\).
For \(k=1\) this follows from \(2^{x^m-x}\succ x\); induction
follows by taking the difference of the preceding exponents.
Thus \(xL_J=o(\ell)\), and the same holds for smaller \(j\).
\end{proof}

Throughout the window induction, frames are restricted to $<T$.
Let \(A_j\) be a single ordinal bounding the Archimedean classes
of all such frames in every interval \([Y,YR]\), \(R\ge1\), with
\(\log R=O(L_j)\), uniformly over \(Y,R\). For the moment this
is a quantity to be bounded, not an assumption at all levels.
Proposition~\ref{shallow} supplies \(A_0=\omega^\omega\), since
\(L_0=x\).

The same bound controls the Archimedean alphabet of the nonzero
frame images modulo the additive subgroup \(O(Y)\). Indeed the
frames \(O(Y)\) vanish; for frames \(D\succ Y\), that quotient
preserves Archimedean equivalence and inequivalence.

At tier $j$ use the four ordinary families of
Section~\ref{l-ordinary-defs} with $L=L_j$. We may write
$\NN_j,\SSS_j,\KK_j,\FF_j$ to display the tier. In particular,
the moving family has the essential condition
$\log(1+w)=O(L_j)$. All depth and center conditions are retained.

\section{Sequential window and ordinal induction}

Define, in this order,
\begin{equation}
 \begin{aligned}
 s_0&=\omega,\\
 c_j&=s_j\hs\omega\hs10,\\
 \nu_j&=\bigl(a_0\hs(c_j\cdot\omega)\bigr)+1,\\
 s_{j+1}&=\omega^{\nu_j+2}\qquad(j<J).
 \end{aligned}
 \label{h-eq-5.1}
\end{equation}

\begin{proposition}\label{h-prop:bootstrap}
For every \(0\le j\le J\):
\begin{enumerate}[label=\textup{(\roman*)},leftmargin=2.5em]
\item every width-\(L_j\) frame alphabet has type at most
\(\omega^{s_j}\);
\item every admissible ordinary tier-\(j\) node, at every coarse
node, has count at most \(\EE(\nu_j)\).
\end{enumerate}
\end{proposition}
\begin{proof}
Assertion (i) at \(j=0\) is Proposition~\ref{shallow}.
Given (i) at tier \(j\), first prove (ii) by induction on the
finite coarse tree. If child primitive counts are \(\le\EE(d)\),
fine frame counts are \(\le\EE(d+2)\). The positive-sum operation
costs
\begin{equation}
 \bigl(\EE(d+2)^\omega\bigr)^{\od\omega^{s_j}}
       =\EE\bigl((d+3)+s_j\bigr).
 \label{h-eq-5.2}
\end{equation}
Here \(s_j>0\), so the outer natural power has a limit exponent.
The finite products and unions for \(\NN_j,\KK_j\), and the
top-frame sum in \(\KK_j\), are absorbed by
\(\EE(d\hs c_j)\); indeed
\[
 ((d+3)+s_j)+2\le d\hs s_j\hs5<d\hs c_j.
\]
Compute \(\SSS_j,\FF_j\) first and
then \(\NN_j,\KK_j\) at the node, as in Lemma~\ref{ordinary-step}.
Thus at coarse height \(h\) every count is at most
\begin{equation}
 \EE\bigl(a_0\hs(c_j\od(h+1))\bigr)\le\EE(\nu_j).
 \label{h-eq-5.3}
\end{equation}
Actual leaves have the stronger bound \(\EE(a_0)\); terminal
nonleaves are covered by the one-step bound with \(d=a_0\).
The last inequality follows from
\(\sup_{k<\omega}c_j\od k=c_j\cdot\omega\).
This proves (ii) uniformly in the root, the finite tree height,
all local parameters, and the constants in the \(O\)-relations.

Only after (ii) has been established do we prove (i) for \(j+1\).
Take any nonempty width-\(L_{j+1}\) frame window and choose an
actual frame \(D_0\) in it. Every frame in the window has
\(\log(D/D_0)=O(L_{j+1})\). Introduce a fresh coarse root
\begin{equation}
 U=V=D_0,\qquad q=x+xL_{j+1},\qquad w=L_{j+1}.
 \label{h-eq-5.4}
\end{equation}
By Lemma~\ref{h-lem:scales}, \(q=O(\ell)\).
The center displacement is zero,
\(w=O(q\theta/x)\), and
\[
 \log(1+w)=O(L_j).
\]
Consequently \(\FF_j(D_0;L_{j+1})\) is admissible and its
family contains the entire frame window. Its quotient is the
Archimedean quotient. The uniform bound in (ii) therefore gives
\begin{equation}
 A_{j+1}\le\EE(\nu_j)
       \le\omega^{\omega^{\nu_j+2}}=\omega^{s_{j+1}}.
 \label{h-eq-5.5}
\end{equation}
The order of proof is (i)\(_j\), then (ii)\(_j\), then
(i)\(_{j+1}\); there is no same-tier assumption of the alphabet
being proved.
\end{proof}

\begin{remark}
The factor \(x\) in \eqref{h-eq-5.4} ensures the required radius bound.
Using merely \(q=x+L_{j+1}\) would generally fail
\(w=O(q\theta/x)\). A uniform bound on the window's center is
not required: ordinary tier \(j\) was established for every
coarse root before \eqref{h-eq-5.4} was introduced.
\end{remark}

\section{The moving deep-band operation \texorpdfstring{$\GG$}{G}}\label{l-sec:G}

Use the deep family defined in Section~\ref{l-ordinary-defs},
and let $J$ be the final window tier. Write $\varepsilon=e^{-p}$ and inflate the $O(w)$ width to
\begin{equation}
 z=\theta w,\qquad
 Y=e^{-p-z}V,\qquad Z=e^zV.
 \label{l-eq-6.21}
\end{equation}
For every member $W$ one has $e^{-z}V\le W\le Z$ eventually.
In its frame decomposition discard $D<Y$ and retain the frames
in $[Y,Z]$.  The discarded finite sum is $O(Y)$ and
$Y\le\varepsilon W$.
If two member values agree modulo additive $O(Y)$, they are
$\equiv_\varepsilon$-equivalent.  Thus the additive quotient is
\emph{finer}, and its size is a valid upper bound for $\GG$.
No choice of a representative of a coarser class is used to define
a finer quotient.

All retained frames belong to $\Fr(\calB)$, since
$p=O(q)$, $z=O(q\theta^2/x)=o(q)$, and
$\log(V/U)=O(t_{\calB})=o(q)$.
Put
\begin{equation}
 \sigma=\theta(p+2z)=O(q\theta),\qquad
 \eta=e^{-p-2z}=Y/Z.
 \label{l-eq-6.22}
\end{equation}
They have a common prefix above $\sigma$.
At logarithmic precision $\eta$ the recursive count coordinate is
\begin{equation}
 \GG(W_f;\sigma/b,\ p+2z+\log b).
 \label{l-eq-6.23}
\end{equation}
Its depth is at least $\log x$ and is $O(q+\log b)$.
Its radius and center satisfy \eqref{l-eq-6.6}.
Let $\mathcal D=\{D:\ D\text{ is a frame},\ Y\le D\le Z\}$
and let $A\ge2$ bound its actual logarithmic quotient at precision
$\eta=Y/Z$.  By Lemma~\ref{l-lem:quotient-maps}(ii), the map
\[
 (\log\mathcal D)/O(\eta)\longrightarrow q_Y(\mathcal D),
 \qquad [\log D]\longmapsto q_Y(D)
\]
is increasing and onto.  Here the representatives are chosen
inside $\mathcal D$, as in Remark~\ref{l-rem:restricted-exp},
so they obey the essential upper bound $D\le Z$.
Consequently $P_Y=q_Y(\mathcal D)\setminus\{0\}$ has type
at most $A$.

The Archimedean alphabet of those positive frame images is bounded
separately, by applying Lemma~\ref{l-lem:coding} to the same frame
family with precision $1$.
This uses \eqref{l-eq-6.19}--\eqref{l-eq-6.20}, now with $\sigma$ from \eqref{l-eq-6.22}.
Their parameters are admissible because
$\log\sigma=O(\log q+\log\theta)=O(\log\ell)=O(L_J)$.
The wide moving coordinates also have $\log(1+\sigma/b)=O(L_J)$.
Thus every alphabet-counting node is admissible at ordinary tier $J$.
Let $F\ge1$ bound the actual frame alphabet
$|\mathcal D/\arch|$ supplied by this separate coding.
Lemma~\ref{l-lem:quotient-maps}(iv) gives
$|P_Y/\arch|\le F$.  Put $\mathcal M=\calL(V;w)$.
Discarding the frames below $Y$ changes a member only by
$O(Y)$, so
\[
 q_Y(\mathcal M)\subseteq\Sigma^0P_Y.
\]
Finally, the lower member bound $Y\le\varepsilon W$ makes
\[
 q_Y(\mathcal M)\longrightarrow\mathcal M/{\equiv_\varepsilon},
 \qquad q_Y(W)\longmapsto[W]_{\equiv_\varepsilon}
\]
a weakly increasing surjection by
Lemma~\ref{l-lem:quotient-maps}(iii).  Restricting to this
actual member quotient before taking the last map and applying
(BM3) therefore give
\begin{equation}
 \GG(V;w,p)\le |q_Y(\mathcal M)|
       \le|\Sigma^0P_Y|\le(A^\omega)^{\od F}.
 \label{l-eq-6.24}
\end{equation}
As before, zero is absorbed by the infinite majorants used in
the reserve calculation.  All deep recursive nodes and all
alphabet-counting nodes are at coarse children.

Define
\begin{equation}
 c'=\omega^{\nu_J+\lambda+3}+a_0+3,\qquad
 g=c'\cdot\omega+a_0\cdot2.
 \label{deep-reserve}
\end{equation}

\begin{proposition}[Uniform deep reserve]\label{l-prop:deep}
Every admissible $\GG$ node is bounded by $\EE(g)$.
\end{proposition}
\begin{proof}
If the child deep counts are bounded by $\EE(d)$, $d\ge a_0$,
the frame count in \eqref{l-eq-6.23} is at most $A=\EE(d+2)$.
Proposition~\ref{h-prop:bootstrap} and the separate alphabet coding
give $F\le\EE(\nu_J+2)$.
This is a limit ordinal, so \eqref{l-eq-6.24} yields
\begin{equation}
 \begin{aligned}
 \GG(V;w,p)
 &\le \bigl(\EE(d+2)^\omega\bigr)^{\od\EE(\nu_J+2)}\\
 &=\EE\bigl((d+3)+\omega^{\nu_J+2}\bigr)
 \le\EE(d\hs c').
 \end{aligned}
 \label{l-eq-7.4}
\end{equation}
The last inequality follows from
$c'\ge\omega^{\nu_J+\lambda+3}>\omega^{\nu_J+2}$.
The actual-leaf bound is $\EE(a_0)$.
Induction on height, including terminal nonleaves as above, gives
\begin{equation}
 \GG(V;w,p)\le
       \EE\bigl(a_0\hs(c'\od(h+1))\bigr).
 \label{l-eq-7.5}
\end{equation}
For every positive integer $k$, $c'\od k\ge a_0$, so
\begin{equation}
 a_0\hs(c'\od k)\le(c'\od k)+a_0\cdot2
                         \le c'\cdot\omega+a_0\cdot2=g.
 \label{l-eq-7.6}
\end{equation}
The second inequality uses
$\sup_k c'\od k=c'\cdot\omega$.
Equations \eqref{l-eq-7.5}--\eqref{l-eq-7.6} prove the claimed uniform bound.
\end{proof}

\section{The uniform local estimates at the two roots}
\begin{theorem}\label{local-roots}
Under \eqref{local-sizes}, at every tower cap considered in
\eqref{h-eq-3.1}, Condition~\ref{g-hyp:fibres} holds with
$n=\nu_J$ and $g$ from \eqref{deep-reserve}.
More precisely, H1 has bound $\EE(g)$ and H2 has bound
$\EE(\nu_J)$, both strictly below their padded reserves.
\end{theorem}

\begin{proof}
For (H2), let $E$ be an Archimedean class and choose $V\in E$.
Take the coarse root
\[
 (U,q)=(V,x+x\log2),\qquad t=x\log2.
\]
All conditions \eqref{l-eq-6.3} hold, so $Q_{\rho_0}(E)=\KK(E;t)$ is
at most $\EE(\nu_J)$ by Proposition~\ref{h-prop:bootstrap}.

For (H1), if $1/b\succeq\rho_0$, the quotient has at most one
element.  Otherwise $b\succ2^x$.
Put $p=\log b$, choose $V\in E$, and take
\[
 (U,q)=(V,x+p),\qquad w=\rho_0.
\]
Here $p\ge\log x$, $p<\ell$, $q=O(\ell)$, and
$w=O(q\theta/x)$.
The band $\calL(V;w)$ is exactly the class $E$.
Thus
$Q_{1/b}(E)=\GG(V;\rho_0,\log b)\le\EE(g)$
by Proposition~\ref{l-prop:deep}.
Empty classes cause no issue.
\end{proof}

\section{The height induction and its ordinal cost}

\begin{theorem}[Finite-tower bounds]\label{h-thm:main}
For every fixed integer \(m\ge2\) and every \(n\ge1\), put
\(r_n=2+n(n+3)/2\). Then
\begin{equation}
 |\Sk_{<E_n(x^m)}|<\omega_{r_n},\qquad
 |\Sk_{<E_n(x^m)}/\arch|<\omega_{r_n-1}.
 \label{h-eq-7.1}
\end{equation}
Consequently \(|\Sk_{<E_n(x)}|<\omega_{10+4^n}\) for every
\(n\ge1\).
\end{theorem}
\begin{proof}
Induct on \(n\), simultaneously for every fixed \(m\ge2\).
For \(n=1\), \(E_1(x^m)=2^{x^m}<2^{2^x}\).
By \cite[Theorem 13.1]{BM},
\[
 |\Sk_{<2^{x^m}}|<\omega_3,\qquad
 |\Sk_{<2^{x^m}}/\arch|<\omega_2.
\]
The total is a proper initial segment of the classical fragment.
For strictness of the quotient, the class of \(2^{x^{m+1}}\)
lies above every cap class and still belongs to that fragment.
These bounds imply \eqref{h-eq-7.1} with \(r_1=4\).

Now fix \(n\ge2\), set \(H=E_{n-1}(x^m)\), and write
\(r=r_{n-1}\). The cap is bounded by the induction hypothesis.
The leaf comes from the same height because
\begin{equation}
 H^\kappa<E_{n-1}(x^{m+1}).
 \label{h-eq-7.2}
\end{equation}
For height one this follows from \(\kappa x^m<x^{m+1}\).
At greater heights, taking a logarithm reduces it to
\(\kappa E_{n-2}(x^m)<E_{n-2}(x^{m+1})\), which follows
from the faster growth of the right-hand tower.
Thus
\[
 |B|,|\Sk_{<H^\kappa}|<\omega_r,\qquad
 |B/\arch|<\omega_{r-1}.
\]
Choose \(a_2,a'<\omega_{r-2}\), with \(a_2\ge\omega^3\),
to bound these totals by \(\EE(a_2),\EE(a')\).
Corollary~\ref{g-cor:alphabet} gives \(\lambda<\omega_{r-1}\).
Equation \eqref{local-a0} then yields \(a_0<\omega_{r-1}\).

At the fixed cap use the tiers in \eqref{h-eq-3.1}, through \(J=n-2\).
From \eqref{h-eq-5.1}, induction on \(j\) gives
\begin{equation}
 \nu_j<\omega_{r+j-1}.
 \label{h-eq-7.3}
\end{equation}
Indeed this is immediate at \(j=0\). At the next tier,
\(s_j=\omega^{\nu_{j-1}+2}<\omega_{r+j-1}\), and the endpoint \(\omega_{r+j-1}\)
is closed under the finite natural sums and the
multiplication by \(\omega\) used in \eqref{h-eq-5.1}.
Equation \eqref{deep-reserve} gives
\[
 c',g<\omega_{r+J}.
\]
All three parameters \(\Gamma,\bar\gamma,\bar\tau\) in
\eqref{g-ledger-parameters} are consequently \(<\omega_{r+J}\), since
\(\omega^{\lambda\cdot2+1}<\omega_r\) and
\(\omega^\lambda<\omega_r\).
Theorems~\ref{local-roots} and~\ref{g-thm:ledger} now give
\begin{equation}
 |\Sk_{<T}/\arch|<\omega_{r+J+2},\qquad
 |\Sk_{<T}|<\omega_{r+J+3}.
 \label{h-eq-7.4}
\end{equation}
As \(J=n-2\), the height recurrence is
\[
 r_n=r_{n-1}+n+1,\qquad r_1=4.
\]
Its solution is \(r_n=2+n(n+3)/2\), proving \eqref{h-eq-7.1}.
Finally \(E_n(x)<E_n(x^2)\) and
\(2+n(n+3)/2<10+4^n\) for every \(n\ge1\).
\end{proof}

\section{The sharper triple-tower specialization}
\begin{proposition}\label{triple-sharp}
For every fixed positive integer $m$,
\[
 |\Sk_{<E_3(x^m)}|<\omega_{10}.
\]
\end{proposition}
\begin{proof}
For $m\ge2$, Theorem~\ref{h-thm:main} at height two gives
$|\Sk_{<E_2(x^m)}|<\omega_7$ and quotient $<\omega_6$.
Set $H=E_2(x^m)$, $T=E_3(x^m)$; the leaf fragment is below
$E_2(x^{m+1})$ and has the same total bound.
Choose $a_2,a'<\omega_5$ and $\lambda<\omega_6$, so that
$a_0<\omega_6$.

Use a single ordinary width scale $L=x^m$. Proposition~\ref{shallow}
provides the alphabet $\omega^\omega$, so the ordinary one-step
argument has $s=\omega$, $c=s\hs\omega\hs10$ and a uniform
reserve $\EE(v)$, where $v=(a_0\hs(c\cdot\omega))+1<\omega_6$.
The proof is exactly the finite-tree induction in
Proposition~\ref{h-prop:bootstrap}, without the successor-window
step. At this cap $\log\ell=O(x^m)$, so the deep argument applies
with this single ordinary tier. Formula~\eqref{deep-reserve},
with $v$ in place of $\nu_J$, gives $g<\omega_7$.
The two root constructions remain admissible and the global
ledger gives total $<\omega_{10}$.
The case $m=1$ follows by inclusion into $m=2$.
\end{proof}

\section{Cofinality and the full Skolem order}
\begin{lemma}\label{cofinal}
The finite towers $E_n(x)$ are cofinal in $\Sk$.
\end{lemma}
\begin{proof}
Induct on a Skolem term. If two subterms are below $A=E_k(x)$,
their sum and product are below $2^A$ eventually. Their power is
at most $A^A<2^{2^A}=E_{k+2}(x)$ eventually. The initial terms
$1,x$ lie below $E_1(x)$. This proves cofinality.
\end{proof}
\begin{corollary}
The preceding bounds imply $|\Sk|=\varepsilon_0$.
\end{corollary}
\begin{proof}
By Lemma~\ref{cofinal}, $\Sk$ is the increasing union of its
fragments below $E_n(x)$. Theorem~\ref{h-thm:main} and inclusion
$E_n(x)<E_n(x^2)$ give
\[
 |\Sk|=\sup_n |\Sk_{<E_n(x)}|
 \le\sup_n\omega_{r_n}=\varepsilon_0.
\]
The reverse inequality is the lower bound recalled in (BM5).
\end{proof}

\section*{Acknowledgement}
The author used contemporary artificial-intelligence tools 
during the preparation of this manuscript. They assisted with exploratory
calculations, checking intermediate arguments, identifying expository
gaps, and preparing successive drafts. All mathematical claims,
citations, and final formulations were examined and approved by the
author, who takes full responsibility for the correctness of the work.
The author acknowledges support by FWO project G0ABD26N.

\end{document}